\documentclass[]{article}

\usepackage[scale=.8]{geometry}
\usepackage{graphicx}
\usepackage{amssymb}
\usepackage{amsmath}
\usepackage{hyperref}
\usepackage{enumerate}
\usepackage{amsthm}
\usepackage{amsmath}
\usepackage{float}
\usepackage{mathtools}
\usepackage{lineno}
\usepackage{bbm}
\usepackage{textcomp}
\usepackage{esint}
\usepackage{cleveref}
\usepackage{xstring}
\usepackage{extarrows}
\usepackage{romannum}

\numberwithin{equation}{section}

\usepackage{xcolor}
\definecolor{lanse}{RGB}{0,0,255} 
\definecolor{zise}{RGB}{112,48,160} 
\definecolor{hongse}{RGB}{200,0,0} 
\usepackage{lineno}

\makeatletter
\renewenvironment{proof}[1][\proofname]{%
	\par\pushQED{\qed}\normalfont%
	\topsep6\p@\@plus6\p@\relax
	\trivlist\item[\hskip\labelsep\bfseries#1\@addpunct{.}]%
	\ignorespaces
}{%
	\popQED\endtrivlist\@endpefalse
}
\makeatother

\theoremstyle{thmstyleone}%
\newtheorem{theorem}{Theorem}
\theoremstyle{thmstyletwo}%
\newtheorem{remark}{Remark}%

\theoremstyle{thmstylethree}%
\newtheorem{definition}{Definition}%

\theoremstyle{thmstylefour}%
\newtheorem{lemma}{Lemma}

\theoremstyle{thmstylefive}%
\newtheorem{assumption}{Assumption}

\theoremstyle{thmstylesix}%

\makeatletter
\def\keywords{\xdef\@thefnmark{}\@footnotetext}
\makeatother

\usepackage{amsfonts}
\usepackage{amsthm}
\usepackage{amsmath}
\usepackage{amssymb}
\usepackage{mathrsfs}
\usepackage{comment}
\usepackage{hyperref}
\usepackage{graphicx}
\usepackage{epstopdf}
\usepackage{bm}
\usepackage{bbm}
\usepackage{cleveref}
\usepackage{xcolor}

\renewcommand{\d}{\mathrm{d}}

\newcommand{\R}{\mathbb{R}}

\newcommand{\N}{\mathbb{N}}
\newcommand{\E}{\mathbb{E}}

\newcommand{\T}{\mathcal{T}}
\renewcommand{\P}{\mathbb{P}}
\newcommand{\ind}[1]{\mathbbm{1}\left\{ #1\right\}}
\newcommand{\abs}[1]{\lfloor #1 \rfloor}

\title{On the maximal displacement of subcritical branching random walk in random environment}
\author{Wenxin Fu and Wenming Hong}

\begin{document}
	
	\pagenumbering{arabic}
	
	\maketitle
	
	\begin{abstract}
		In this paper, we consider the subcritical branching random walk in a random environment. We assume the branching and the step jump are independent; and the branching is in random envirenment, i.e., the particles in generation $n$ produce children according the probability measure $F_n\in \mathcal{P}\left(\N_0\right)$, and the $F_n$, $n=1,2,\cdots, $ are i.i.d  under the $P_E$. ``subcritical" means that	$
		a:=\E[X_1]\in (-\infty,0)$, where $X_1:=\log \overline{F}_1$ and $\overline{F}_1$ is the mean of $F_1$.
		
		Denote by $M$ the maximal displacement of the branching system. We find that the tail probability is heavily influnced by the log-Laplace transform of environment $
		\Lambda_e(\rho) = \log \E \left[ e^{\rho X} \right] <\infty.
		$ and the log-Laplace transform of step jump $h$, $\Lambda_s(\lambda):=\log \E_\mu\left[ e^{\lambda h}\right]$.

		For quenched probability estimation, we proved that under mild assumption, for $P_E$ almost surely environment $\xi$,
		\begin{equation*}
			\lim\limits_{n\rightarrow\infty}
			\frac{1}{x}\log P_\xi\left(M\geqslant x\right) = -\lambda_0,
		\end{equation*}
		where $\lambda_0$ is the unique solution of $\Lambda_s(\lambda_0)+\Lambda_e^\prime(0+)=0$. 
		
		For annealed probability estimation, let $\alpha:=\sup\left\{\rho:\Lambda_e(\rho)\leqslant 0\right\}$, $\varTheta(\rho):=\Lambda_e^\prime(\rho) + \Lambda_s(\lambda_\rho) -\lambda_\rho \Lambda_s^\prime(\lambda_\rho)$, in which $\lambda_\rho$ is the solution of equation $\Lambda_e(\rho)+\rho\Lambda_e(\lambda)=0$.
		Under some mild condition, we have the following three cases 
		
		\noindent (1) when  $\alpha>1$ and $\varTheta(1)<0$,  
		\begin{equation*}
			\lim\limits_{x\rightarrow\infty}e^{\lambda_1 x }\P(M\geqslant x ) = C_1;
		\end{equation*}
	    (2) when $\alpha>1$ and $\varTheta(1)=0$, 
	    \begin{equation*}
	    	\lim\limits_{x\rightarrow\infty}\sqrt{x}e^{\lambda_1 x }\P(M\geqslant x ) = C_2;
	    \end{equation*}
	    (3) when $\alpha>1$ and $\varTheta(1)>0$ or $\alpha\leqslant 1$, there exists a unique $\rho\in (0,1)$ such that $\varTheta(\rho)=0$, and some finite positive constant $C_3$ and $C_4$ such that
	    \begin{equation*}
	    	C_3 x^{-\frac{3}{2}} \leqslant e^{\rho\lambda_\rho x }\P(M\geqslant x ) \leqslant C_4.
	    \end{equation*}
	    
	   To deal with the challenges caused by the random environment, we formulated a ``position" indexed random work $R_x$ from the excursion of the spine. 
	    
	\end{abstract}
	
	\section{Introduction and Main Results}

	The maximal displacements of spatial branching systems have been investigated by many authors.
	For supercritical branching random walk ($m > 1$),
	it can be traced back to Hammersley \cite{Hammersley},
	Kingman \cite{Kingman}, Biggins \cite{Biggins1976} and Bramson \cite{Bramson}.
	Since then extensively studies on this topics have appeared in recent years, see for example
	\cite{Addario2009,Aidekon2013,Bachmann,Bramson2016,Bramson2009,Hu2009} 
	and references therein.
	
	For critical or subcritical cases, the system dies out eventually, 
	so the maximal displacement of the system is finite almost surely. 
	Thus, it is natural to consider the tail probability of the maximal displacement. 
	
	The asymptotic law for the maximal displacement of a critical branching Brownian motion can trace back to Sawyer and Fleischman \cite{Fleischman1979Maximun}.
	For the critical branching random walk, see \cite{Kesten1995,Lalley2015} and references therein.
	In \cite{Lalley2015}, Lalley and Shao introduced a discrete Feynman-Kac formula, and proved that under mild moment condition
	\begin{equation}\label{1}
		\lim\limits_{x\rightarrow\infty}x^2\P(M\geqslant x) = \frac{6\eta^2}{\sigma^2}
	\end{equation}
	where $\eta^2$ is the variance of jump distribution and $\sigma^2$ is the variance of offspring distribution. Moreover, they proved a conditional limit theorem, which shows
	\begin{equation}\label{2}
		\mathcal{L}\left(\left.
		\frac{M_n}{\sqrt{n}} \right|Z_n>0
		\right) \Longrightarrow \mathcal{L}(G),
	\end{equation}
	where $G$ is some random variable related to a super-Brownian motion.
	
	For subcritical cases, by using the discrete Feynman-Kac formula, in \cite{Neuman2017}, Neuman and Zheng obtain the asymptotic law of the maximal displacement. Recently, in \cite{Fu2025}, we reconsidered the subcritical branching random walk, and used the spine decomposition and optional line method refined the results. The results show that
	\begin{equation}
		\lim\limits_{x\rightarrow\infty}e^{\lambda x}\P(M\geqslant x) \quad
		\text{exists and belongs in} \quad (0,1].
	\end{equation}
	where $\lambda $ is the unique positive solution of the equation $\Lambda(\lambda)+\log m=0$, and $\Lambda(\lambda)$ is the log-Laplace exponent of jump distribution, and $m\in (0,1)$ is the mean of offspring distribution.
	
	\
	
	In this article, we are interested in a branching random walk in random environment (BRWre), i.e., a branching random walk in time-inhomogeneous environment, which has been introduced by Biggins and Kyprianou in \cite{Biggins2004}.
	For supercritical branching random walk in random environment, the results about extreme processes are obtained by Huang and Liu in \cite{Huang2014} and Mallein and Mi\l o\'s in \cite{Mallein2019}.
	
	In recent work \cite{Fu2026}, we investigated a kind of critical branching random walk in random environment (such system will die out eventually). We proved that
	conditional on survival event $\{Z_n>0\}$, 
	\begin{equation}\label{4}
		\mathcal{L}\left(\left.\frac{M_n}{n^{\frac{3}{4}}}\right|Z_n>0\right) \Longrightarrow \mathcal{L}\left(A_{\Lambda}\right),
	\end{equation}
	where $M_n$ is the maximal displacement at generation $n$, and $A_\Lambda$ is a random variable determined by some Brownian meander process. And based on above conditional limit theorem, we also proved that there exist some constants $C_1$ and $C_2$ such that
	\begin{equation}\label{5}
		0< C_1
		\leqslant \liminf\limits_{x\rightarrow\infty} x^{\frac{2}{3}}\P(M>x)
		\leqslant \limsup\limits_{x\rightarrow\infty} x^{\frac{2}{3}}\P(M>x) \leqslant C_2<\infty,
	\end{equation}
	where $M$ is the maximal displacement of the whole system. 
	Compared with the constant environment case (i.e. the \cref{1,2}), 
	it reveals that, the conditional limit speed for $M_n$ in random environment (i.e., $n^{3/4}$) is significantly greater than that of constant environment case (i.e., $n^{1/2}$), and so is the tail probability for $M$ (i.e., $x^{-2/3}$ vs $x^{-2}$).
	The phenomenon stems from the essential effect induced by the random environment. The random environment allows the branching system, conditional on survival, to have more particles survive, which in turn leads to heavier tail behavior for the maximum position.

	In this paper, we consider a kind subcritical branching random walk in random environment, and study the asymptotic behavior of the tail probability for the corresponding maximum displacement.
	
	\
	
	Following the Ulam-Harris notations for trees, we write
	\begin{equation*}
		\mathrm{U}^*= \bigcup\limits_{n\in \N} \mathbb{N}^n, \quad\text{and}\quad
		\mathrm{U} = \mathrm{U}^*\cup \emptyset.
	\end{equation*}
	the set of finite sequences of integers, with the convention $\mathbb{N}^0 = \emptyset$, where $\emptyset$ is the
	sequence of length 0, which encodes the root of the trees we consider.
	
	Let $\nu = (\nu_1,\nu_2,\cdots,\nu_n)\in \mathrm{U}$, then $\nu$ represents the $\nu_n^{th}$ child of the $\nu_{n-1}^{th}$ child
	of ... of the $\nu_1^{th}$ child of the initial individual $\emptyset$.
	We write $|\nu|$ the generation to
	which $\nu$ belongs, with convention $|\emptyset| = 0$. And for any $1\leqslant k\leqslant n$, we set $\nu(k) = (\nu_1,\nu_2,\cdots,\nu_k)$ the $k^{th}$ ancestor of $\nu$, specially $\nu(0)=\emptyset$.
	For $\varsigma,\nu \in \mathrm{U}$, we write $\varsigma \prec \nu$, if there exists $k\leqslant |\nu|$ such that $\varsigma = \nu(k)$, or in other words, $\varsigma$ is an ancestor of $\nu$.
	
	A plane rooted tree $\mathcal{T}$ is a subset of $\mathrm{U}$ which satisfies the following three properties:
	\begin{itemize}
		\item the root $\emptyset\in\mathcal{T}$. 
		\item if $\nu\in\mathcal{T}$ and $\mu\neq \emptyset$, then $\nu(|\nu|-1) \in \mathcal{T}$.
		\item if $\nu \in \mathcal{T}$, there exists an integer $ p(\nu)$, such that $ (\nu,k)\in \mathcal{T} $ for any $1\leqslant k\leqslant p(\nu)$; and if $p(\nu)=0$, then $\nu$ has no offspring in $\mathcal{T}$.
	\end{itemize}
	In fact, $p(\nu)$ is the number of children of $\nu$ in $\mathcal{T}$.
	A plane rooted marked tree is a pair $(\mathcal{T},V)$, where $\mathcal{T}$ is a plane rooted marked tree and $V: \mathcal{T}\longmapsto \R$ is a function from the nodes of tree $\mathcal{T}$ to $\R$.
	
	Let $\mathcal{P}\left(\N_0\right)$ 
	be the space of probability measures on 
	$\N_0:=\left\{0,1,\cdots\right\}$.
	Equipped with the metric of total variation,
	$\mathcal{P}\left(\N_0\right)$ becomes a Polish space. 
	For any probability measure $F\in \mathcal{P}\left(\N_0\right)$, we also use $F$ to denote the generating function of this probability measure on $\N$, the mean value and normalized second factorial moment of $F$ is denoted as
	\begin{equation}\label{mean and stand-var}
		\overline{F}:= \sum_{k=0}^{\infty} k F[k],
		\quad\tilde{F}:=\dfrac{1}{\overline{F}^2} \sum_{k=1}^{\infty}k(k-1)F[k].
	\end{equation}
	Also, denote by $ \varkappa(F,a) $ 
	the standardized truncated second moment of 
	the probability measure $F\in \mathcal{P}\left(\N_0\right)$,
	\begin{equation}\label{6}
		\varkappa(F,a):= \dfrac{1}{\overline{F}^2} \sum_{y=a}^{\infty} y^2F[y].
	\end{equation}
    And for probability measure $F \in \mathcal{P}\left(\N_0\right)$ such that $\overline{F}<\infty$, we denote by $F^S \in \mathcal{P}\left(\N_0\right)$ the size-biased probability measure of $F$. And in fact, the generating function of $F^S$ is given by
    \begin{equation}\label{size-biased distribution}
    	F^S(s) = \sum_{k=1}^{\infty} \frac{kF[k]}{\sum_{j=1}^{\infty} jF[j]} s^k = s\frac{F^{\prime}(s)}{F^{\prime}(1)}, \quad \forall s\in [0,1] .
    \end{equation}
    
	Let $F$ be a random variable 
	taking values in $\mathcal{P}\left(N_0\right)$. A two-side infinite sequence $\xi:= \left\{\cdots, F_{-1},F_{0},F_1,\right.$ $\left.F_2,\cdots\right\}\in \mathcal{P}\left(\N_0\right)^{\mathbb{Z}} $
	of i.i.d. copies of $F$ is said to form a random environment, and we use $P_E$ denote the corresponding probability of the environment.
	For any $k \in \mathbb{Z}$, we define
	\begin{equation}\label{quenched random walk}
		X_k:=\log \overline{F}_k,\quad \eta_k:=\tilde{F}_k,
		\quad S_0=0,\quad \text{and} \quad S_n=S_{n-1}+X_n.
	\end{equation}
	Define by $ \mathcal{F} = \sigma(\xi) $ the $\sigma$-field generated by the two-side environment $\xi$.

	For any $k\in \mathbb{Z}$, denote by $ T_k $ the shift operator $T_k:\mathcal{P}\left(\N_0\right)^{\mathbb{Z}} \longmapsto \mathcal{P}\left(\N_0\right)^{\mathbb{Z}} $ which shifts the environment $\xi$ by the vector $k$, i.e.
	\begin{equation}\label{shift operator}
		\forall l\in \mathbb{Z}, \quad (T_k\xi)(l) = \xi(k+l) = F_{k+l}.
	\end{equation}
	And denote by $\Gamma_k $ the truncation operator $\Gamma_k:\mathcal{P}\left(\N_0\right)^{\mathbb{Z}} \longmapsto \mathcal{P}\left(\N_0\right)^{\mathbb{N}} $, such that
	\begin{equation}\label{cut off operator}
		\Gamma_k\xi = 
		\left\{
		\xi(k+1),\xi(k+2),\cdots
		\right\}
		=
		\left\{F_{k+1},F_{k+2},\cdots
		\right\} \in \mathcal{P}\left(\N_0\right)^{\mathbb{N}}.
	\end{equation}
	
	Given an environment $\xi\in \mathcal{P}\left(\N_0\right)^{\mathbb{Z}}$ (or conditioned on environment $\xi$), we consider a special time-inhomogeneous branching random walk in environment $\xi$, which is a plane rooted marked tree $(\mathcal{T},V)$ and can be mathematical description as follows:
	\begin{itemize}
		\item It starts at root $\emptyset$ located at original, i.e. $V(\emptyset) = 0$.
		\item For any $k\geqslant1$, given a family $\left\{p(\nu):~\nu\in\mathrm{U},~|\nu|=k-1\right\}$ of i.i.d random variables with the distribution given by $F_k$.
		\item For different $k$, these families are independent with each other.
		\item The plane rooted tree $\mathcal{T}$ is given by $ \mathcal{T} = \left\{
		\nu\in\mathrm{U}:~\text{for all}~k\leqslant |\nu|, ~\nu_k\leqslant
		p(\nu(k-1)) \right\}.$
		\item For any $\nu\neq\emptyset\in \mathcal{T}$, the movement of $\nu$ respect to its parent is given by ``$V(\nu)-V(\nu(|\nu|-1))$". Conditioned on genealogy tree $\mathcal{T}$, all of these movements are i.i.d random variables with distribution given by $\mu(\d x)$.
	\end{itemize}
	Indeed, this branching system can be constructed as follow.
	We consider a family of independent random elements
	$\{ (h_{\nu}^1,h_{\nu}^2,\cdot,h_{\nu}^{\rho(\nu)}):\nu\in\mathrm{U} \}$, where $p(\nu)$ has the law of $F_{|\nu|+1}$, and conditioned on $p(\nu)$, $\{h_{\nu}^i:1\leqslant i\leqslant p(\nu)\}$ are i.i.d random variables with distribution given by $\mu$. 
	And then, the plane rooted tree which represents the genealogy of the population is given by $ \mathcal{T} = \left\{
	\nu\in\mathrm{U}:~\text{for all}~k\leqslant |\nu|, ~\nu(k)\leqslant\right.$ $\left.
	p(\nu(k-1)) \right\}$. We set $V(\emptyset) = 0$, and for any $\nu=(\nu_1,\nu_2,\cdots,\nu_{k})\in\mathcal{T}$ with $|\nu|=k$, 
	\begin{equation}
		V(\nu) := V(\nu(k-1))+ h_{\nu(k-1)}^{\nu_k} = \sum_{s=0}^{k-1} h_{\nu(s)}^{\nu_{s+1}}.
	\end{equation}
	
	The pair $\left(\T,V\right)$ is called the branching random walk 
	in the time-inhomogeneous environment $\xi$.
	For each $n\geqslant0$, define $Z_n:= \# \{u\in \T,|u|=n\}$ the number of particles survived in generation $n$.
	Conditioned on environment $\xi$, we use $P_\xi$ to denote the law of the branching random walk $\left(\T,V\right)$ in environment $\xi$, which is also called the quenched probability.
	
	\begin{remark}
		According to above construction, the quenched distribution of $\left(\T,V\right)$ only depend on the environment on positive half-axis, i.e. $\xi^+ :=\left(F_1,F_2,\cdots\right)$. However, to obtain a structure consistent with the subsequent parts, we directly introduce the two-sided environment.
	\end{remark}
	
	The quenched law $P_\xi$ can be regarded as the probability kernel. Therefore, combining with $P_E$ and $P_\xi$, 
	there exists a probability measure $\P(\d x, \d \xi)=P_\xi(dx) P_E(\d \xi)$ such that
	\begin{equation}\label{annealed prob}
	\iint f(x,\xi) \P(\d x, \d \xi) = \iint f(x,\xi) P_\xi(\d x)P_E(\d \xi)
	\end{equation}
	for all positive and measurable real-valued function $f$.
	The joint probability of the environment $\xi$ and branching system $(\mathcal{T},V)$ is just given by the probability measure $ \P(\d x, \d \xi)=P_\xi(dx) P_E(\d \xi)$, which is called the annealed probability. Abuse use of notations, we also use $P_E\otimes P_\xi$ to represent the annealed probability.

	Under annealed probability $\P$, $\{X_k;k\geqslant1\}$ is 
	a sequence of $i.i.d$ copies of logarithmic of the mean of the offspring number $X:=\log\overline{F} $,
	and the sequence $ S:=\left(S_0,S_1,\cdots\right) $ 
	is called the associated random walk.
	
	\begin{assumption}\label{subcritical}
		We assume that
		\begin{equation}
			a:=\E[X_1]\in (-\infty,0).
		\end{equation}
		and for any $\rho\geqslant0$, we have
		\begin{equation}
			\Lambda_e(\rho) = \log \E \left[ e^{\rho X} \right] <\infty.
		\end{equation}
	And define \begin{equation}\label{def of alpha}
	 \alpha:=\sup\left\{\rho:\Lambda_e(\rho)\leqslant 0\right\}.
	\end{equation}
	\end{assumption}
	
	Under Assumption \ref{subcritical}, we know $\Lambda_e^\prime(0)<0$, and the branching system is subcritical, thus will die out eventually. Denote by
	\begin{equation}
		M_n:= \max\{V(\nu):\nu\in \mathcal{T}, |\nu|=n\}, \quad \text{and}
		\quad M:= \sup\limits_{n\geqslant 0} M_n.
	\end{equation}
	the maximal displacement of the system at generation $n$ or the whole maximal displacement of the system. Therefore, $M$ and $M_n$ are finite almost surely.
	
	\begin{remark}\label{R1}
		For subcritical branching processes in random environment, the extinction probability exhibits a phase transition depending on the sign of $\Lambda_e^\prime(1)$. Thus the subcritical regimes can be divided into three cases: strongly subcritical ($\Lambda_e^\prime(1)<0$), intermediately  subcritical ($\Lambda_e^\prime(1)=0$) and weakly subcritical ($\Lambda_e^\prime(1)>0$). We refer readers to \cite{B2} for more discussions.
	\end{remark}
	
	Besides the branching environment, the jump distribution $\mu(\d x)$ also plays an important role in the spatial branching system. And for the jump distribution $\mu(\d x)$, we need following assumption.
	\begin{assumption}\label{Jump Assumption}
		The jump distribution $\mu(\d x)$ is  supported on $\mathbb{Z}$, non-trivial, satisfies $\mu([2,\infty)) = 0$ with mean $0$. Thus, for any $\lambda>0$, denote by
		\begin{equation}
			\Lambda_s(\lambda):=\log \E_\mu\left[ e^{\lambda h}\right]<\infty,
		\end{equation}
	where under $\E_\mu$, $h$ is a random variable with distribution given by $\mu$. 
	\end{assumption}
	
	Under Assumption \ref{Jump Assumption}, we know $\Lambda_s^\prime(0)=0$, for any $\delta>0$, the equation $\Lambda_s(\lambda) =\delta $ has a unique solution in $(0,\infty)$. And we denote by $\kappa(\delta)$the unique solution of the equation.
	
	\begin{remark}\label{R2}
		Under \Cref{Jump Assumption}, the random walk $\left(H_0=0,H_1,\cdots\right)$ with one step distribution $\mu(\d x)$ doesn't experience the overshoot, which implies that for any positive integer $x$, we have $H_{\tau_x}=x$, where $\tau_x:= \inf\{k\geqslant 0:H_k\geqslant x\}$. Without overshoot, it is helpful for us to use the coupling method to investigate some random variables. For general random walks, the approach proposed in this paper is still applicable in principle. The only remaining issue to address is the treatment of the overshoot. Referring to our previous work \cite{Fu2025}, we may extend the relevant results by virtue of the property that the overshoot of a supercritical random walk converges in distribution.
	\end{remark}
	
	For quenched probability estimates, as time tends to infinity, the quenched tail probability decay of the maximal displacement $M$ is determined by the mean $a:=\E[X_1]\in (-\infty,0)$ (defined in \cref{subcritical}) of environment and the jump distribution.
	\begin{theorem}\label{Quenched theorem}(Quenched probability decay)
		Under \Cref{subcritical} and \Cref{Jump Assumption}, with addition technical assumption
		\begin{equation}\label{addition assumption for quenchde probability}
			\E\left[ \eta_1 \right] <\infty.
		\end{equation}
	    where $\eta$ is defined in \cref{quenched random walk}.
		Then, for $P_E$ almost surely environment $\xi$, we have
		\begin{equation}
			\lim\limits_{n\rightarrow\infty}
			\frac{1}{x}\log P_\xi\left(M\geqslant x\right) = -\lambda_0,
		\end{equation}
		where $\lambda_0:= \kappa(-a)$, $\kappa(\cdot)$ is the inverse function of $\Lambda_s(\lambda)$, and $a = \E[X_1]$ is the constant defined in \Cref{subcritical}.
	\end{theorem}
	
	For annealed probability estimates, similar to the extinction probability (see \Cref{R1}), the tail probability also undergoes three different phases. However, the corresponding phase transition points are not consistent. In fact, the phase transition depends on both environment and movement.
	
	By using many-to-one method, optional line and an essential coupling argument, by choosing a suitable parameter $\lambda>0$, we related the annealed probability decay to a random walk $\{R_x,x\geqslant 0\}$ under some probability measure $ P_E\otimes P_{\overline{Y}} $, such that
	\begin{equation}\label{18}
		R_x :=  \sum_{i=-T_x+1}^{0} X_i + T_x\Lambda_s(\lambda)-\lambda x,
	\end{equation}
	where $\{T_x;x\geqslant0\}$ is a sequence of random times with drift $\left(\Lambda_s^\prime(\lambda)\right)^{-1}$ and independent with $\{X_i;i\in \mathbb{Z}\}$, and $\lambda$ is some constant to be determined, see Section 2 for rigorously 
	definitions.
	And by the \Cref{Couping Lemma} in Section 2, we have 
	\begin{equation}
		\P(M\geqslant x) = P_E\otimes P_{\overline{Y}} \left(e^{R_x}B_x\right)
	\end{equation}
	where $B_x$ is some random variable defined in \cref{def of overlineA}.
	
	\begin{remark}
		For the subcritical branching random walk in constant environment, the $ X_i $ in \cref{18} is equal to $ \log m$ (where $m\in (0,1)$ is the mean of offspring distribution), by choosing $\lambda$ such that $\Lambda_s(\lambda)+\log m = 0$ (see \cite{Fu2025} for detail discussions), the random walk $\{R_x, x\geqslant 0\}$ reduces to a trivial path.
		And for random environment models, the random walk $\{R_x,x\geqslant 0\}$ is nontrivial, which leads to the key challenges posed by the random environment models.
	\end{remark}

	According to above equality, if $\alpha>1$ (where $\alpha$ is  defined in \cref{def of alpha}), by choosing $\lambda= \lambda_1$ (i.e. the solution of $\Lambda_e(1)+\Lambda_s(\lambda)=0$), the $\{e^{R_x + \lambda_1 x};x\geqslant 0\}$ is the exponent martingale with parameter $\rho=1$. Via the standard measure change of exponential martingale, we can construct a probability measure $Q_E\otimes P_{\overline{Y}}$ (defined in Section 3), such that \begin{equation} e^{\lambda_1 x}P_E\otimes P_{\overline{Y}} \left(e^{R_x}B_x\right) = Q_E\otimes P_{\overline{Y}}(B_x). \end{equation}
	It should be noted that the behavior of the sequence of random variables $\{B_x,x\geqslant 0\}$ depends strongly on the properties of the random walk $\{R_x,x\geqslant 0\}$ under $Q_E\otimes P_{\overline{Y}}$, especially the drift of the random walk. By simple calculation, the mean of random walk $\{R_x,x\geqslant 0\}$ under $Q_E\otimes P_{\overline{Y}}$ is given by $\frac{\varTheta(1)}{\Lambda_s^{\prime}(\lambda_1)}$ (also see \cref{mean of R}), and the sign of the drift is equal to the sign of $\varTheta(1)$, where
	\begin{equation}
		\varTheta(1):=\Lambda_e^\prime(1)+\Lambda_s(\lambda_1)-\lambda_1\Lambda_s^\prime(\lambda_1)
	\end{equation}
	
	Under $\alpha>1$, if $\varTheta(1)<0$, then the random walk $\{R_x,x\geqslant0\}$ under $Q_E\otimes P_{\overline{Y}}$ has negative drift, which leads to the sequence of random variables $\{B_x,x\geqslant0\}$ converge to some positive random variable $B_\infty$ \cref{Lemma of Strongly}, which lead to that the decay of $\P(M\geqslant x)$ is exactly given by $e^{-\lambda_1x}$.
	For such cases, the following theorem describe the decay of annealed tail probability of maximal displacement $M$.
	\begin{theorem}\label{Theorem of strongly cases}
		Under \Cref{subcritical} and \Cref{Jump Assumption}, assume $\alpha>1$ and $\varTheta(1)<0$, which implies that $\Lambda_e(1)<0$ and 
		\begin{equation}\label{19}
			\varTheta(1) = \Lambda_e^\prime(1)+\Lambda_s(\lambda_1)-\lambda_1\Lambda_s^\prime(\lambda_1) <0,
		\end{equation}
		where $\lambda_1$ satisfies $ \Lambda_e(1)+\Lambda_s(\lambda_1) =0 $.
		In addition, we impose the following integrability condition
		\begin{equation}\label{addition assumption for strongly cases}
			\E\left( e^{X_1-\Lambda_e(1)} \log_+ \eta_1\right)<\infty,
		\end{equation}
		where $\log_+ x = \log (\max\{x,1\})$.
		Then there exists some positive constant $C_1 \in (0,1]$, such that
		\begin{equation}
			\lim\limits_{x\rightarrow\infty}e^{\lambda_1 x }\P(M\geqslant x ) = C_1.
		\end{equation}
	\end{theorem}
	
	\begin{remark}
		The subcritical branching random walk in a constant environment can be regarded as a special case of $\alpha>1$ and $\varTheta(1)>0$ cases subcritical branching random walk in random environment (i.e. the \cref{19} is always satisfied for $\Lambda_e(\lambda)=\lambda \log m$).
		From this perspective, \Cref{Theorem of strongly cases} also indicates the corresponding result for branching random walks in constant environment.
	\end{remark}
	
	Under $\alpha>1$, if $\varTheta(1)=0$, the random walk $\{R_x,x\geqslant 0\}$ under $Q_E\otimes P_{\overline{Y}}$ is recurrent. For such cases, besides the exponential decay $e^{-\lambda_1 x}$, there is also an additional polynomial decay.
	\begin{theorem}\label{Theorem of intermediately cases}
		Under \Cref{subcritical} and \Cref{Jump Assumption}, assume $\alpha>1$ and
		$\varTheta(1)=0$, which implies that $\Lambda_e(1)<0$ and 
		\begin{equation}\label{22}
			\varTheta(1)=\Lambda_e^\prime(1)+\Lambda_s(\lambda_1)-\lambda_1\Lambda_s^\prime(\lambda_1) =0.
		\end{equation}
		where $\lambda_1$ satisfies $ \Lambda_e(1)+\Lambda_s(\lambda_1) =0 $.
		In addition, we assume that there exists some $\delta>0$ such that
		\begin{equation}\label{additional assumption for intermediately}
			\E\left(e^{X_1-\Lambda_e(1)} (\log_+ \eta_1)^{2+\delta}\right)<\infty.
		\end{equation}
		Then there exists some finite and positive constant $C_2$, such that
		\begin{equation}
			\lim\limits_{x\rightarrow\infty}\sqrt{x}e^{\lambda_1 x }\P(M\geqslant x ) = C_2.
		\end{equation}
	\end{theorem}

	\begin{remark}
		For the $\alpha>1$ and $\varTheta(1)=0$ cases subcritical branching random walk in random environment, the condition \cref{22} is essential, which implies that the system necessarily admits a nontrivial random environment. Under this condition, besides the exponential decay, for the further polynomial decay, the environment needs to be more favorable for the system to reach larger heights. This is the reason for the emergence of the additional $\sqrt{x}$ decay.
	\end{remark}
	
	Under $\alpha>1$, if $\varTheta(1)>0$, the leading exponential decay of $\P(M\geqslant x)$ is not given by $e^{-\lambda_1 x}$ any more.
	To obtain the probability estimate for such cases, we should introduce a reference function $\varTheta(\rho)$.
	
	For any $\rho\in (0,\alpha]$, we have $\Lambda_e(\rho)\leqslant 0$. Thus, there exists a unique finite and positive number $\lambda_\rho$ such that $\Lambda_e(\rho)+ \rho \Lambda_s(\lambda_\rho)=0$.
	The function $\varTheta(\rho)$ is defined on $(0,\alpha]$, such that
	\begin{equation}
		\varTheta(\rho)
		= \Lambda_e^\prime(\rho) + \Lambda_s(\lambda_\rho) -\lambda_\rho \Lambda_s^\prime(\lambda_\rho), \quad \forall ~ \rho \in (0,\alpha].
	\end{equation}
	
	\begin{remark}
		About the equation $\Lambda_e(\rho)+ \rho \Lambda_s(\lambda_\rho)=0$, it represents the balance between the cost required for a measure transformation on the environment and the cost required for a measure transformation on the jump distribution, thereby enabling the extraction of the dominant order of the tail probability of the maximal displacement. Strictly speaking, for the random variable $R_x$ (see \cref{def of R_x}), we should obtain $P_E\otimes P_{\overline{Y}}\left[ e^{\rho(R_x+\lambda_\rho x)} \right] = 1$, then the leading order of tail probability of $M$ is exactly given by $e^{\rho \lambda_\rho x}$.
	\end{remark}
	
	About the function $\varTheta(\rho)$, following lemma provides some necessary properties.
	\begin{lemma}\label{Classify}
		Under \Cref{subcritical}and \Cref{Jump Assumption},
		the function $\varTheta(\rho)$ is finite on $(0,\alpha]$, where $\alpha:=\sup\left\{\rho:\Lambda_e(\rho)\leqslant 0\right\}$, and is a strictly increasing function on $(0,\alpha]$. Moreover, we have 
		\begin{equation}
			\lim\limits_{\rho\rightarrow 0^+} \varTheta(\rho)<0.
		\end{equation}
		And if $\alpha<\infty$, we have
		\begin{equation}
			\varTheta(\alpha)>0,
		\end{equation}
		which leads that in $ (0,\alpha] $, the exist a unique zero point of $\varTheta(\rho)$.
	\end{lemma}
	
	\begin{proof}
		Obviously, under \Cref{subcritical}and \Cref{Jump Assumption}, according to the definition of $\varTheta(\rho)$, it is finite on $(0,\alpha]$.
		Calculating the deviation of $\varTheta(\rho)$, we have
		\begin{equation}\label{deviation of varTheta}
			\varTheta^\prime(\rho)
			= \Lambda_e^{\prime\prime}(\rho)-\lambda_\rho \frac{\d \lambda_\rho}{\d \rho}\Lambda_s^{\prime\prime}(\lambda_\rho).
		\end{equation}
		Under \Cref{subcritical} and \Cref{Jump Assumption}, according to the convexity of log-Laplace exponent, for any $\rho>0$, we have
		$ \Lambda_e^{\prime\prime}(\rho)>0$ and $ \Lambda_s^{\prime\prime}(\lambda_\rho)>0 $.
		Also, for $\rho\in(0,\alpha]$,  by the convexity and $\Lambda_e(0)=0$, we know that $ \frac{\Lambda_e(\rho)}{\rho} $ is a increasing function about $\rho$. Then, due to the relationship between $\lambda_\rho$ and $\rho$ (i.e. the equation $\Lambda_e(\rho)+ \rho \Lambda_s(\lambda_\rho)=0$), for $\rho\in (0,\alpha]$, $\lambda_\rho$ is a decreasing function about $\rho$. Therefore, we have $\frac{\d \lambda_\rho}{\d \rho}\leqslant 0$.
		Combining with \cref{deviation of varTheta}, for any $\rho\in (0,\alpha]$, we have
		\begin{equation}
			\varTheta^\prime(\rho)>0,
		\end{equation}
		which ensures that $\varTheta(\rho)$ is strictly increasing on $(0,\alpha]$.
		
		For $\rho\rightarrow 0^+$, under \Cref{subcritical}, we have
		\begin{equation}
			\lim\limits_{\rho\rightarrow 0^+} \frac{\Lambda_e(\rho)}{\rho}
			=\Lambda_e^{\prime}(0)<0
		\end{equation}
		And by the continuous of $\Lambda_s(\lambda)$ and its inverse function $ \kappa(\delta) $, we have
		\begin{equation}
			\lambda_0:=\lim\limits_{\rho\rightarrow 0^+} \lambda_\rho = \kappa(-\Lambda_e^\prime(0))>0
		\end{equation}
		which satisfies $\Lambda_s(\lambda_0) = -\Lambda_e^{\prime}(0)$.
		Then, be the continuity of corresponding function, we have
		\begin{equation}
			\lim\limits_{\rho\rightarrow 0^+} \varTheta(\rho)
			=
			\Lambda_e^{\prime}(0) +  \Lambda_s(\lambda_0)
			-\lambda_0 \Lambda_s^\prime(\lambda_0)
			= -\lambda_0 \Lambda_s^\prime(\lambda_0)<0.
		\end{equation}
		
		And if $\alpha<\infty$, under \Cref{subcritical}, we have $\Lambda_e(\alpha)=0$ and $\Lambda_e^\prime(\alpha)>0$, then by the relationship $\Lambda_e(\rho)+ \rho \Lambda_s(\lambda_\rho)=0$, we have $\lambda_\alpha = 0$, thus
		\begin{equation}
			\varTheta(\alpha) = \Lambda_e^\prime(\alpha)>0.
		\end{equation}
		
		So far, \Cref{Classify} is proved.
	\end{proof}
	
	Note that, the function $\varTheta(\rho)$ doesn't require $\alpha>1$. And in fact, relying on the \Cref{Classify}, Case $\alpha>1$ and $\varTheta(1)>0$ and Case $\alpha\leqslant 1$ can be considered uniformly together.
	By \Cref{Classify}, if $\alpha>1$ and $\varTheta(\alpha)>0$ or $\alpha<1$, for both cases, there exists a unique $\rho\in(0,1)$ such that
	\begin{equation}
		\varTheta(\rho)= \Lambda_e^\prime(\rho) + \Lambda_s(\lambda_\rho) -\lambda_\rho \Lambda_s^\prime(\lambda_\rho) = 0,
	\end{equation}
	where $\lambda_\rho$ is the unique positive solution of equation $\Lambda_e(\rho)+\rho\Lambda_s(\lambda_\rho) =0$.
	For such cases, the exponential decay of $\P(M\geqslant x)$ is given by $ e^{-\rho\lambda_\rho x} $. The following theorem precisely characterizes the exponential decay of the annealed tail probability for $M$, and also providing a rough estimate of the additional polynomial decay.
	
	\begin{theorem}\label{Theorem of weakly cases}
		Under \Cref{subcritical} and \Cref{Jump Assumption}, assume that $\alpha>1$ and $\varTheta(1)>0$ or $\alpha<1$, by \Cref{Classify}, for both cases, there exists a unique $\rho\in(0,1)$ such that
		\begin{equation}\label{25}
			\varTheta(\rho)= \Lambda_e^\prime(\rho) + \Lambda_s(\lambda_\rho) -\lambda_\rho \Lambda_s^\prime(\lambda_\rho) = 0,
		\end{equation}
		where $\lambda_\rho$ is the unique positive solution of equation $\Lambda_e(\rho)+\rho\Lambda_s(\lambda_\rho) =0$.
		In addition, we assume that there exists some $\delta>0$ such that
		\begin{equation}\label{additional assumption for weakly}
			\E\left(e^{\rho X_1-\Lambda_e(\rho)} (\log_+ \eta_1)^{2+\delta}\right)<\infty.
		\end{equation}
		Then there exist some finite and positive constants $C_3$ and $C_4$, such that for any $x\geqslant 0$, we have
		\begin{equation}
			C_3 x^{-\frac{3}{2}} \leqslant e^{\rho\lambda_\rho x }\P(M\geqslant x ) \leqslant C_4.
		\end{equation}
	\end{theorem}
	
	\begin{remark}
		For $\alpha>1$ and $\varTheta(1)>0$ or $\alpha\leqslant 1$  cases subcritical branching random walk in random environment, the condition \cref{25} is essential. Under this condition, for annealed probability estimates, the maximum position can reach the target height by choosing the non-typical environments, whose associated cost is lower. This ultimately leads to a difference in the exponential decay of the annealed probability compared with other cases.
		
		However, owing to the absence of a crucial and sharp upper bound inequality on the harmonic moments (i.e. the conditional expectation of $B_x$, which is defined in \cref{def of overlineA}), we are unable to obtain the exact polynomial decay. Nevertheless, we conjecture that the exact polynomial order should be given by $x^{-\frac{3}{2}}$.
	\end{remark}

	Based on \Cref{Theorem of strongly cases,Theorem of intermediately cases,Theorem of weakly cases}, we classify the subcritical branching random walk in three cases.
	\begin{definition}\label{Def of classfy}
		About the subcritical branching random walk in random environment (BRWRE in short) $(\mathcal{T},V)$,
		
		(1) if $\alpha>1$ and $\varTheta(1)<0$, then $(\mathcal{T},V)$ is said to belong into Class \Romannum{1} subcritical BRWRE.
		
		(2) if $\alpha>1$ and $\varTheta(1)=0$, then $(\mathcal{T},V)$ is said to belong into Class \Romannum{2} subcritical BRWRE.
		
		(3) if $\alpha>1$ and $\varTheta(1)>0$ or $\alpha\leqslant1$, then $(\mathcal{T},V)$ is said to belong into Class \Romannum{3} subcritical BRWRE.
	\end{definition}
	
	\begin{remark}
		If $\alpha>1$, depending on the sign of $\varTheta(1)$, the above classification is complete.
		However, it is possible that $\alpha$ is less than or equal to 1. 
		For the branching random walk in random environment falling into this case, it should note that the measure change on environment with exponent parameter $\rho=1$ is excessive, which share the same intrinsic structure as the $\alpha>1$ and $\varTheta(1)>0$ case.
		So, we categorized $\alpha\leqslant 1$ as the Class \Romannum{3} as well. And in the latter part of the paper, Class \Romannum{3} admits a relatively unified treatment.
		Moreover, the above classification incorporating both $ \alpha$ and the sign of $\varTheta(1)$ is exhaustive under \Cref{subcritical} and \Cref{Jump Assumption}.
	\end{remark}
	
	\begin{remark}
		If $ \Lambda_e^\prime(1)\leqslant 0 $, i.e. the classical strongly or intermediately subcritical BPRE (see \cite{B2}), regardless of the specific structure of the jump distribution, we always have $\alpha>1$ and $\varTheta(1)<0$, thus the corresponding subcritical BRWRE belongs into Class \Romannum{1} cases.
		And if $ \Lambda_e^\prime(1)>0 $, i.e. the classical weakly subcritical BPRE, the corresponding structure becomes much more complicated, and all three classes outlined above can occur.
	\end{remark}
	
	\begin{remark}
		The classification comes form the affection of random environment. For the subcritical branching random walk in constant environments, see \cite{Neuman2017,Fu2025}, the tail probability of the maximum position exhibits only exponential decay, whereas for subcritical branching random walk in random environment, its asymptotic behavior is far more intricate. This profoundly reflects the essential impacts induced by random environments.
	\end{remark}

	The rest of the paper is organized as follow.
	In \Cref{Section2.1,Section2.2}, we provide some preliminaries.
	And in \Cref{Section2.3}, we construct an annealed coupling, which plays an important role in investigating annealing probability decay. In \Cref{Section4}, we prove our main results.
	
	\section{Preliminaries}\label{Section2}
	
	Given the environment $ \xi$, under $P_\xi$, the branching system $(\mathcal{T},V)$ is a typical inhomogeneous branching random walk. Conditional on the first generation of the branching process, it is not hard to verified that the quenched probability $P_\xi(M\geqslant x)$ satisfies a typical non-linear stochastic difference equation. For differential equations, we can have the Feynman-Kac formula to formally construct the solutions. However, in the discrete case, similar tools may need to be further refined. For example, in \cite{Lalley2015}, they introduced a discrete Feynman-Kac formula and investigate the maximal displacement of critical branching random walk, and later in \cite{Neuman2017} and so on.
	It should be noted that,
	for the branching system, the recently developed size-biased measure change method may be a relatively effective approach. Also, inspired by our previous work \cite{Fu2025} about homogeneous subcritical branching random walk, we can use the size-biased measure change and optional line to express the quenched probability $P_\xi(M\geqslant x)$.
	
	\subsection{Some Results About Random walk}\label{Section2.1}
	
	Assume that $ \{H_0:=0,H_1,\cdots\}$ is a random walk with one step distribution $\mu(\d y)$ (i.e. $\{H_k-H_{k-1}:k\geqslant 1\}$ are i.i.d random variables with the distribution $\mu(\d y)$), and we just use $Q$ denote the probability.
	Define by $\mathcal{G}_n:= \sigma(H_0,H_1,\cdots,H_n)$ the natural filtration of the random walk, and set $\mathcal{G}_\infty = \sigma\left(\bigcup\limits_{n=1}^\infty\mathcal{G}_n\right) $.
	
	For any $\lambda > 0$, due to the martingale property of $ \{e^{\lambda H_n-n\Lambda_s(\lambda)}:n\geqslant 0\} $ (i.e the exponent martingale), by Kolmogorov consistency theorem, there exists a probability $ Q_\lambda $ defined on $\mathcal{G}_\infty$, such that
	for any $n\geqslant 1$, \begin{equation}\label{key}
		\frac{\d Q_\lambda}{\d Q}\big|_{\mathcal{G}_n} = e^{\lambda H_n-n\Lambda_s(\lambda)}.
	\end{equation}
	Under \Cref{Jump Assumption}, for $\lambda >0$, under $Q_\lambda$, $\{H_0,H_1,\cdots\}$ is a random walk with positive drift $\Lambda_s^{\prime}(\lambda)$.
	Thus, for any positive integer $x$, the first hitting time $ \tau_x:=\inf\{k\geqslant1:H_k=x\} $ satisfies
	\begin{equation*}
		\tau_x<\infty,\quad  Q_\lambda \quad a.s.
	\end{equation*}
	And we are interested about the probability $Q_\lambda(\tau_x = n)$.
	
	\begin{lemma}\label{Hitting time of Random walk}
		Under \Cref{Jump Assumption}, for any positive integers $x,n$ and $\lambda>0$, We have
		\begin{equation*}
			Q_\lambda(\tau_x = n)
			=Q(\tau_x = n) e^{\lambda x-n\Lambda_s(\lambda)}.
		\end{equation*}
	\end{lemma}
	
	\begin{proof}[Proof of \Cref{Hitting time of Random walk}]
		In fact, $\{\tau_x = n\} \in \sigma(H_0,H_1,\cdots,H_n)$.
		According to relationship between $Q_\lambda$ and $Q$, we have
		\begin{equation*}
			Q_\lambda(\tau_x = n)
			=
			Q (e^{\lambda H_n-n\Lambda_s(\lambda)},\tau_x=n).
		\end{equation*}
		And under \Cref{Jump Assumption}, we know that $H_{\tau_x} = x,~ Q~a.s.$.
		Hence, the lemma is proved.
	\end{proof}

	\subsection{Martingale, Optional line and Measure change}\label{Section2.2}
	
	In this subsection, by measure change and optional line methods,  we can obtain a useful expression about the probability of target event $\{M\geqslant x\}$ for time-inhomogeneous branching random walk.
	
	Given an environment $\xi \in \mathcal{P}\left(\N_0\right)^{\mathbb{Z}}$ such that $\lim\limits_{n\rightarrow\infty} S_n=-\infty$, where  $\{S_n;n\in \mathbb{Z}\}$ are defined in \cref{quenched random walk}.
	Under the quenched probability $P_\xi$,
	the branching system $(\mathcal{T},V)$ is a time-inhomogeneous branching random walk. For any $\lambda>0$, define $W_0:=1$ and for any $n\geqslant 1$,
	\begin{equation}\label{additive martingale}
		W_n(\lambda):= \sum_{\nu\in N_n} e^{\lambda V\left(\nu\right)-n\Lambda_s(\lambda)-S_{n}}.
	\end{equation}
	It is not hard to verify that under $P_\xi$, $\{W_n(\lambda),n\geqslant0\}$ is a martingale of the time-inhomogeneous branching system $(\mathcal{T},V)$. 
	
	For any positive integer $x>0$, define the optional line $\mathcal{L}_x$ such that
	\begin{equation}\label{Def of Ln}
		\mathcal{L}_x(\nu)=1 \quad \text{if and only if} \quad 
		V(\nu)\geqslant x ~ \text{and} ~V(\nu(k))<x ~\text{for any}~ k<|\nu|.
	\end{equation}
	Although $\mathcal{L}_x$ is a function on the nodes, it will be often convenient to identify $\mathcal{L}_x$ with the set of nodes where the function takes the value $1$, and we also call $\mathcal{L}_x$ the optional line. Abuse use of notations, we also use $\mathcal{L}_x$ to represent the set $\{\nu\in \mathcal{T} :\mathcal{L}_x(\nu) = 1\}$. 
	Intrinsically, $\mathcal{L}_x$ is the set of vertexes in $\mathcal{T}$ which hit $[x,\infty)$ for the first time.
	Note that the optional line $\mathcal{L}_x$ is a very simple line 
	(i.e. for any $\nu$, the value of $\mathcal{L}_x(\nu)$ is measurable with respect to $\sigma\left(\left\{X(\nu(i)),0\leqslant i \leqslant |\nu|\right\}\right)$, which means that whether $\nu$ is on the line or not is determined by looking the trajectory of its ancestries).
	Readers can see the conception about optional line in \cite{Biggins2004} and references therein.
	
	Next, for any $n\in \N$, let $\mathcal{L}_x^n$ be the function corresponding to the line
	formed by the member of $\mathcal{L}_x$ 
	in the first $n$ generations and the $n$-th generation nodes 
	with no ancestor in $\mathcal{L}_x$, i.e. for any $ \nu\in \mathcal{T} $, \begin{equation}\label{Lnx}\mathcal{L}_x^n(\nu) = 1 \quad\text{if and only if}\quad \begin{cases*}
			|\nu|\leqslant n\quad \text{and} \quad \mathcal{L}_x(\nu) = 1;
			\\
			|\nu| = n \quad\text{and} \quad  X(\nu(k)) \in (-\infty,x), ~\forall k \leqslant |\nu|.
	\end{cases*} \end{equation}
	Using above optional lines, together with additive martingale (eq. \eqref{additive martingale}),
	according to Lemma 6.1, 14.1 and 14.3 in \cite{Biggins2004},
	for any $\lambda>0$, we define (set $\sum_{\emptyset} := 0 $)
	\begin{equation}\label{optional line}
		W_{\mathcal{L}_x}(\lambda)
		:=
		\sum\limits_{\nu\in \mathcal{T}} 
		\mathcal{L}_x(\nu) e^{\lambda V(\nu)-\left|\nu\right|\Lambda_s(\lambda)-S_{|\nu|}}, \quad \E[W_{\mathcal{L}_x}(\lambda)] = 1.
	\end{equation}
	\begin{equation}\label{optional line martingale}
		W_{\mathcal{L}_x^n}(\lambda)
		:=
		\sum\limits_{\nu\in \mathcal{T}} 
		\mathcal{L}^n_x(\nu)
		e^{\lambda V(\nu)-\left|\nu\right|\Lambda_s(\lambda)-S_{|\nu|}},\quad
		\E[W_{\mathcal{L}_x^n}(\lambda)] = 1.
	\end{equation}
	\begin{equation}\label{martingale l1 convergence}
		\{W_{\mathcal{L}_x^n}(\lambda)\}_{n\geqslant 0}~~
		\text{is a martingale and converges in mean to}~~
		W_{\mathcal{L}_x}(\lambda).
	\end{equation}
	
	According to above relationship, for $P_\xi$ almost surely, we have
	$\{M\geqslant x\} = \{\exists \nu\in \mathcal{T}:V(\nu)\geqslant x\} = \{\mathcal{L}_x\neq \emptyset\} = \{W_{\mathcal{L}_x}>0\}.$ Therefore, we know
	\begin{equation}\label{tail probability equality relation}
		P_\xi(M\geqslant x) = P_\xi(|\mathcal{L}_x|>0) = P_\xi\left(W_{\mathcal{L}_x}(\lambda)>0\right),
	\end{equation}
	where $|\mathcal{L}_x|$ is the size of the set $\{\nu\in \mathcal{T} :\mathcal{L}_x(\nu) = 1\}$.

	Due to the non-negative martingale property, we can use the closed martingale $\{W_{\mathcal{L}_x^n}(\lambda)\}_{n\geqslant 0}$ to construct a new probability $\tilde{P}_\xi$, such that
	\begin{equation}\label{measure change}
		\d \tilde{P}_\xi = W_{\mathcal{L}_x}(\lambda) ~\d P_\xi.
	\end{equation}
	Thus, according to \cref{tail probability equality relation,measure change}, we have
	\begin{equation}\label{2.10}
		P_\xi (M\geqslant x) = P_\xi (W_{\mathcal{L}_x}(\lambda)>0) = \tilde{P}_\xi\left[\frac{1}{W_{\mathcal{L}_x}(\lambda)}\right]
	\end{equation}
	And under the probability measure $\tilde{P}_\xi$ (the well-known size-biased distribution), the branching system behaves as follow:
	\begin{itemize}
		\item In generation 0, single particle (denoted by $\omega_0$) locates at $0$, and it is the initial spine particle.
		\item For any $k\geqslant1$, in generation $k$, 
		the spine particle survived at generation $k-1$ (denoted by $\omega_{k-1}$) reproduces according to $F^S_k$ (the size-biased distribution of $F_k$, which is defined in \cref{size-biased distribution}). Then, uniformly choose one child as the next spine particle (denoted by $\omega_k$), and this spine particle moves according to $ e^{\lambda y-\Lambda_s(y)} \mu(\d y)$. Except spine particle, other newborn particles are normal, and move according to $\mu(\d y)$.
		\item For any $k\geqslant1$, in generation $k$, except spine particle, normal particles behave as usual (i,e. they reproduce according to $F_k$ and all children move independently according to jump distribution $\mu(\d x)$).
	\end{itemize}
	
	Denote by $\{V(\omega_0):=0,V(\omega_1),\cdots\}$ the trajectory of spine particles. According to above construction, we know that the distribution of $ \{V(\omega_0):=0,V(\omega_1),\cdots\} $ under $\tilde{P}_\xi$ is the same as the distribution of $\left\{H_0,H_1,\cdots\right\}$ under $Q_\lambda$ (defined in \Cref{Section2.1}) and is independent with the environment $\xi$.
	For any integer $x\geqslant1$, denote by \begin{equation}\label{first hitting time}
		\tau_x:=\inf\{k\geqslant0:V(\omega_k)\geqslant x\}
	\end{equation} the first hitting time of the trajectory.
	Let $\mathcal{G}:= \sigma\left(\left\{V(\omega_0):=0,V(\omega_1),\cdots\right\}\right)$ by the $\sigma$ field generated by the spine's trajectory.
	And under \Cref{Jump Assumption}, the random walk with one step distribution $\mu(\d x)$ doesn't experience overshoot. Therefore, for any $\nu\in\mathcal{T}$ such that $\mathcal{L}_x(\nu) =1$, we have $V(\nu) = x$.
	
	According to above construction, under the size-biased probability measure $\tilde{P}_\xi$, we have
	\begin{align}
		W_{\mathcal{L}_x} &= e^{\lambda V(\omega_{\tau_x})-\tau_x\Lambda_s(\lambda)-S_{\tau_x} } +\sum_{k=0}^{\tau_x-1} \sum_{ \nu \in \mathcal{T}_{\omega_k}, \mathcal{L}_x(\nu)=1 } 
		e^{\lambda V(\nu)-|\nu|\Lambda_s(\lambda) -S_{|\nu|} } \nonumber
		\\
		&= e^{\lambda x-\tau_x\Lambda_s(\lambda)-S_{\tau_x} } +\sum_{k=0}^{\tau_x-1} \sum_{ \nu \in \mathcal{T}_{\omega_k}, \mathcal{L}_x(\nu)=1 } e^{\lambda x-|\nu|\Lambda_s(\lambda) -S_{|\nu|} }
		\nonumber\\
		&= e^{\lambda x-\tau_x\Lambda_s(\lambda)-S_{\tau_x} }
		+ \sum_{k=0}^{\tau_x-1} e^{\lambda V(\omega_k)- k\Lambda_s(\lambda)-S_{k}}
		\Phi(\mathcal{T}_{\omega_k},V(\omega_k),x)
		\nonumber\\
		&=
		e^{\lambda x -\tau_x\Lambda_s(\lambda)-S_{\tau_x}}
		\left[
		1+\sum_{k=0}^{\tau_x-1} e^{\lambda (V(\omega_k)-x)- (k-\tau_x)\Lambda_s(\lambda)-(S_{k}-S_{\tau_x})}
		\Phi(\mathcal{T}_{\omega_k},V(\omega_k),x)
		\right]
		. \label{spine decomposition}
	\end{align}
	where $\mathcal{T}_{\omega_k}:=\{\omega_k\}\bigcup\{\nu \in \mathcal{T}: |\nu|>k,~ \nu(k)=\omega_k ~ \text{and} ~ \nu(k+1)\neq \omega_{k+1} \} $ is the subtree rooted at $\omega_k$ and constructed only by the normal particles (i.e. excluded the subtree rooted at $\omega_{k+1}$), and
	\begin{equation}\label{37}
		\Phi(\mathcal{T}_{\omega_k},V(\omega_k),x) := \sum_{ \nu \in \mathcal{T}_{\omega_k}, \mathcal{L}_x(\nu)=1 } e^{\lambda \left(x-V(\omega_k)\right)-(|\nu|-k)\Lambda_s(\lambda) -(S_{|\nu|}-S_k) }.
	\end{equation}
	
	\begin{remark}
		Note that, under $\tilde{P}_\xi$, for any $\nu\in\mathcal{T}$ such that $\mathcal{L}_x(\nu)=1$, either $\nu$ is the spine particle $\omega_{\tau_x}$, or $\nu$ is some normal particle and can trace back to $\omega_k$ for some $ 0\leqslant k<\tau_x$, which implies that $\nu\in \mathcal{T}_{\omega_k}$. And \cref{spine decomposition} is the realization of such observation.
	\end{remark}
	
	\begin{lemma}\label{Lemma of spine decomposition}
		Given the environment, under $\tilde{P}_\xi$, conditional on $\mathcal{G}$, for admissible 
		$k$ (i.e. $0\leqslant k\leqslant\tau_x-1$), the conditional distribution of $\Phi(\mathcal{T}_{\omega_k},V(\omega_k),x)$ depends only on the pair $\left(x-V(\omega_k),\Gamma_k(\xi) \right)$, where $\Gamma_k$ is defined in \cref{cut off operator}. Moreover, for different admissible $k$, the random variables $\{\Phi(\mathcal{T}_{\omega_k},V(\omega_k),x);0\leqslant k\leqslant \tau_x-1\}$ are conditional independent with each other.
	\end{lemma}
	
	\begin{proof}[Proof of \Cref{Lemma of spine decomposition}]
		According to the construction of $(\mathcal{T},V)$ unde $\tilde{P}_\xi$, conditional on $\mathcal{G}$, the conditional distribution of the subtree $\mathcal{T}_{\omega_k}$ is determined by $\Gamma_k(\xi)$. Also, by the space homogeneity and the definition of $\Phi$, the lemma is obvious.
	\end{proof}
	
	Combining with \cref{2.10,spine decomposition}, we have
	\begin{align}
		P_\xi(M\geqslant x)
		&=
		\tilde{P}_\xi\left[\frac{1}{W_{\mathcal{L}_x}(\lambda)}\right]
		\nonumber\\
		&=
		\tilde{P}_\xi\left[e^{S_{\tau_x}+\tau_x\Lambda_s(\lambda)-\lambda x}\left(1+\sum_{k=0}^{\tau_x-1} e^{\lambda (V(\omega_k)-x)- (k-\tau_x)\Lambda_s(\lambda)-(S_{k}-S_{\tau_x})}
		\Phi(\mathcal{T}_{\omega_k},V(\omega_k),x) \right)^{-1} \right]
		\nonumber\\
		&=
		\tilde{P}_\xi\left[ e^{R_x} B_x \right]\label{quenched express}
	\end{align}
	where $R_x:= S_{\tau_x}+\tau_x\Lambda_s(\lambda)-\lambda x $ and
	$B_x:= \left(1+\sum\limits_{k=0}^{\tau_x-1} e^{\lambda (V(\omega_k)-x)- (k-\tau_x)\Lambda_s(\lambda)-(S_{k}-S_{\tau_x})}
	\Phi(\mathcal{T}_{\omega_k},V(\omega_k),x) \right)^{-1} $. Note that,  $B_x$ is a positive and bounded random variable (in fact, $0<B_x\leqslant1$).
	
	By \cref{quenched express}, we construct the quenched size-biased probability distribution $\tilde{P}_\xi$, and use some associated random variables such as $R_x$ and $B_x$ to reformulate the quantity $P_\xi(M\geqslant x)$. Similarly to \cref{annealed prob}, there exists a probability measure $\tilde{\P}$ such that
	\begin{equation}
		\iint f(x,\xi) \tilde{\P}(\d x, \d \xi) = \iint f(x,\xi) \tilde{P}_\xi(\d x)P_E(\d \xi),
	\end{equation}
	and abuse use of notations, we also use $P_E\otimes \tilde{P}_\xi$ to represent $\tilde{\P}$.

	\subsection{An Annealed Coupling}\label{Section2.3}
	
	Note that, under $P_E\otimes \tilde{P}_\xi$, besides the two-side environment $\xi$, there is a trajectory of spine particles, denote by $\{V(\omega_0),\cdots,V(\omega_{\tau_x})\}$. For fixed $\lambda>0$, in annealed probability $P_E\otimes \tilde{P}_\xi$,  the trajectory and the environment are independent with each other. Furthermore, conditioned on the trajectory and the environment, the full branching system can be generated according to a suitable rule. This insight inspires us to consider constructing an annealed coupling. Also, inspired by our previous work \cite{Fu2025}, we can look the size-biased branching system standing at time $\tau_x$ and position $x$, and desire to find a suitable coupling.
	
	It is worth mentioning that for homogeneous branching systems, due to homogeneity, 
	no matter  the specific value of $\tau_x$, the branching mechanism remains identical when looking forward and backward from time $\tau_x$, therefore,
	it is not very hard to reconstruct the branching system (see \cite{Fu2025} for more details).
	However, for random environment setting, the time inhomogeneity brings an essential change, making it infeasible to construct a coupling directly under the quenched probability measure $\tilde{P}_\xi$.
	Therefore, we consider constructing a suitable coupling in the annealed sense.

	The construction of the annealed coupling is divided into four steps. 
	
	\begin{itemize}
		\item First, we specify a two-sided random environment. More precisely, we characterize the probabilistic structure of the environment as viewed from time $\tau_x$.
		\item Second, we determine the trajectory of the spine particles. This part is constructed mainly by the excursion theory.
		\item Third, we describe the whole branching system. To do such, we use the language of transfer measure (or can be translated into the language of conditional expectation). Formally, to obtain the whole branching system, conditional on environment and the trajectory of spine particles, we attach some subtrees to the spine.
		\item Finally, based on above constructions, we introduce the annealed coupling measure and several $\sigma$ fields. After that, we adopt the excursion theory and regard each excursion as a time scale. Then we treat space as a new time index and construct the random walk $\{R_x,x\geqslant 0\}$. Also, from the whole branching system, we construct a sequence of random variables $ \{B_x,x\geqslant0\} $. Moreover, we use these random variables to obtain a coupling lemma (i.e. \Cref{Couping Lemma}), which can be used to express the annealed probability.
	\end{itemize}  
	
	\noindent\textbf{Environment:}

	Under $P_E\otimes \tilde{P}_\xi$, the random time $\tau_x$ is measurable with respect to $\mathcal{G}$ (i.e. the $\sigma$ field generated only by the trajectory of spine particles) and independent with $\mathcal{F}_\infty:=\sigma(\xi)$ (i.e. the $\sigma$ field generated only by environment $\xi $). Thus, under the annealed probability, if we look the system at time $\tau_x$, the shifted environment $T_{\tau_x}\xi$ has the same distribution as $\xi$. Therefore, the annealed probability distribution of the environment viewed from time $\tau_x$ is exactly given by $P_E$.
	The reason for constructing a two-sided environment lies in that the behavior of the subtrees attach to the spine constructed subsequently is indeed affected by the environment before $\tau_x$ and after $\tau_x$.

	Therefore, we generate a two side random environment $\xi$ according to $P_E$, and denote by $\mathcal{F}_\infty:=\sigma(\xi)$ the $\sigma$ field generated by the environment.
	
	\
	
	\noindent\textbf{The Trajectory of Spine Particles:}
	
	Recall that, under $\tilde{P}_\xi$, the size-biased branching system has a spine. When looking the trajectory of spine at time $\tau_x$ and position $x$, the trajectory can be reconstructed by excursion theory. Moreover, for different $x$, these trajectories admit a unified treatment in distribution.
	
	Let $\{H_0,H_1,\cdots\}$ be a random walk with one step distribution $e^{\lambda y-\Lambda_s(\lambda y)} \mu(\d y)$. Define $\theta:= \inf\{k\geqslant1:H_k=1\}$, and we have \begin{equation}\label{mean of theta}\E[\theta] = \left( \int ye^{\lambda y-\Lambda_s(\lambda y)} \mu(\d y) \right)^{-1} = \Lambda^\prime_s(\lambda)^{-1}<\infty.\end{equation} From the excursion $\left(H_{j},0\leqslant j\leqslant \theta\right)$, we denote by
	\begin{equation}\label{excursion}
		\varepsilon:=\left(H_{\theta-j}-H_{\theta}=H_{\theta-j}-1,0\leqslant j\leqslant \theta\right)
	\end{equation}
	the inverse excursion.
	
	Let $\{\varepsilon_m:= (\varepsilon_m(j),0\leqslant j\leqslant\theta_m);m\geqslant1\}$ be a sequence of independent copies of $\varepsilon$. We construct an infinite trajectory $\{Y_0,Y_1,\cdots\} $ by connecting these excursions one by one. For a strict mathematical description, for any $x\geqslant1$, we define
	\begin{equation}\label{def of Tx}
		T_x: = \sum_{j=1}^{x}\theta_j,
	\end{equation}
	and set $T_0:=0$. Then the process
	\begin{equation}\label{infinity trajectory}
		Y_k:= -x+\varepsilon_{x+1}(k-T_x), \quad \text{for} \quad T_x<k\leqslant T_{x+1},
	\end{equation}
	with $Y_0:=0$, is the infinite trajectory what we construct.
	Above construction can trace back to Tanaka’s construction \cite{Tanaka1989}.
	From the infinity path $\textbf{Y}:=\{Y_0,Y_1,\cdots\}$, for any $k\leqslant0$, we define
	\begin{equation*}
		\overline{Y}_{k} := Y_{-k},
	\end{equation*}
	then we can obtain a infinity path $\overline{\textbf{Y}}:=\{\cdots,\overline{Y}_{-2},\overline{Y}_{-1},\overline{Y}_0=0\}$ defined on non-positive time.
	Also, according to above construction, for any integer $x\geqslant 0$, we have
	\begin{equation}\label{27}
		T_x = \sup\{k\geqslant 0: Y_k = -x\} = -\inf\{k\leqslant 0:\overline{Y}_k = -x\}.
	\end{equation}
	Abuse use of notations, denote by $\mathcal{G}_x:=\sigma\left(\{\varepsilon_{k}:1\leqslant k \leqslant x\}\right)$ and  $\mathcal{G}:=\sigma\left(\overline{\textbf{Y}}\right)$, then $T_x$ is $\mathcal{G}_x$ measurable random variable. Also, denote by $P_{\overline{Y}}$ the corresponding probability of $\overline{\textbf{Y}}$.
	
	According to above construction, also see \cite{B2,Tanaka1989}, we have
	
	\begin{lemma}\label{Tanaka's construction}
		The distribution of $\left(\overline{Y}_{-T_x},\overline{Y}_{1-T_x},\cdots,\overline{Y}_{0}\right)$ under $P_{\overline{Y}}$
		is the same as the distribution of $(H_0-x,H_1-x,\cdots,H_{\tau_x}-x)$ under $Q_\lambda$, and the distribution of $\left(V(\omega_0)-x,V(\omega_1),\cdots,V(\omega_{\tau_x})-x\right)$ under $\tilde{P}_\xi$.
	\end{lemma}
	
	By \Cref{Tanaka's construction}, the infinite trajectory $\overline{\textbf{Y}}$ constructed herein coincides with the spine trajectory under the original measure, differing only in the observation perspective in time and position.
	
	\
	
	\noindent\textbf{Attach the subtrees to the spine (in the language of transfer measure):}
	
	Given a two-side environment $\xi$ and an infinity trajectory $\overline{\textbf{Y}}$ (we use $P_{\xi,\overline{Y}}$ denote the transfer probability described as follow), for any $k\leqslant -1$, at time $k$, we drive a time-inhomogeneous branching random walk which is rooted at $(k,\overline{Y}_k)$ and denoted by $\mathcal{T}_k$. The branching system $\mathcal{T}_k$ behaves as follow:
	
	\begin{itemize}
		\item Initially, at generation $k$, a single particle is located at $\overline{Y}_{k}$. 
		\item Initial particle dies, and then reproduces several children, whose distribution is the same as $D_k-1$, where $D_k$ is a random variable with distribution given by $F^S_{k+1}$ (the size-biased distribution of $F_{k+1}$, which is defined in \cref{size-biased distribution}).
		After that, every new particle moves according to the jump distribution $\mu(\d x)$ respectively to its parent.
		\item In generation $t$ ($t> k$), every particle behaves as usual, i.e. reproduces independently according to $F_{t+1}$, and offspring particles jump independently according to $\mu(\d x)$.
	\end{itemize}
	
	From any $k\leqslant-1$, we just attach $\mathcal{T}_k$ to the infinite path $\overline{Y}$ at time $k$. Formally, after that, we connect these subtrees together and form a bigger branching system $\mathcal{T}:= \bigcup\limits_{k=-\infty}^{-1} \mathcal{T}_k $.
	For any $\nu \in \mathcal{T}$, denote by $|\nu|$ the time where $\nu$ survives in, 
	and for any $s<|\nu|$, denote by $\nu(s)$ the only ancestor particle of $\nu$ at generation $s$.
	Define an optional line $\mathcal{L}_0$ such that for any $\nu\in\mathcal{T}$,
	\begin{equation}
		\mathcal{L}_0(\nu) = 1 \quad \text{if and only if} \quad 
		V(\nu) = 0 ~\text{and}~ V(\nu(s)) <0,\forall s<|\nu|.
	\end{equation}
	For any $ k\leqslant -1$, with the optional line $\mathcal{L}_0$, define
	\begin{equation}\label{def of Phik}
		\Phi(\mathcal{T}_k,\overline{Y}_k,0):= \sum\limits_{\nu \in \mathcal{T}_k, \mathcal{L}_0(\nu)=1} e^{-\lambda\overline{Y}_k-(|\nu|-k)\Lambda_s(\lambda)-(S_{|\nu|}-S_k)}.
	\end{equation}

	\begin{lemma}\label{mean of Phi}
		Given $\xi$ and $\overline{Y}$, any $k\leqslant -1$, under $P_{\xi,\overline{Y}}$, the distribution of $\Phi(\mathcal{T}_k,\overline{Y},0)$ depends only on the pair $(\overline{Y}_k,\Gamma_k(\xi))$, where $\Gamma_k$ is defined in \cref{cut off operator}, and for different $k$, these random variables are independent with each other.
		And if the pair $ (\overline{Y}_k,\Gamma_k(\xi)) $ is equal to the pair $ (H_{\omega_k}-x,\Gamma_k(\xi)) $ in \Cref{Lemma of spine decomposition}, then these two random variables has the same conditional distribution.
		Moreover, we have
		\begin{equation}\label{eq.mean of phi}
			P_{\xi,\overline{Y}} \left[\Phi(\mathcal{T}_k,\overline{Y}_k,0)\right] = \eta_{k+1}.
		\end{equation}
	\end{lemma}
	
	\begin{proof}[Proof of \Cref{mean of Phi}]
		According to the construction of $\mathcal{T}_k$ in this subsection and $\mathcal{T}_{\omega_k}$ in \cref{37}, the first part of the lemma is obvious.
		
		About \cref{eq.mean of phi}, for $\mathcal{T}_k$, for any $t\geqslant k+1$, we define
		\begin{equation}
			\overline{W}_k(t):=\sum_{\nu\in \mathcal{T}_k,|\nu|=t}
			e^{\lambda(V(\nu)-\overline{Y}_k) - (t-k)\Lambda_s(\lambda)-(S_t-S_k)}.
		\end{equation}
		
		By the branching property,
		it is not hard to verified that $\{\overline{W}_k(t);t\geqslant k+1\}$ is an additive martingale with respectively to the natural filtration generated by the branching system $\mathcal{T}_k$.
		And according to the branching mechanism of $\mathcal{T}_k$ at time $k$, we have
		\begin{equation}
			P_{\xi,\overline{Y}} [\overline{W}_k(k+1)] = \eta_{k+1}.
		\end{equation}
		
		Note that $\mathcal{L}_0$ is a simple optional line for $\mathcal{T}_k$. Therefore,
		according to Lemma 6.1, 14.1 and 14.3 in \cite{Biggins2004} (i.e. the same arguments appeared in \Cref{Section2.2}),
		and we have \[\Phi(\mathcal{T}_k,\overline{Y}_k,0) = \overline{W}_{\mathcal{L}_0} = \sum_{\nu\in\mathcal{T}_k} \mathcal{L}_0(\nu) e^{\lambda(V(\nu)-\overline{Y}_k)-(|\nu|-k)\Lambda_s(\lambda)-(S_{|\nu|}-S_k)}. \]
		Hence, for $\lambda>0$, we have
		\[ \quad P_{\xi,\overline{Y}} [\Phi(\mathcal{T}_k,\overline{Y}_k,0)] = P_{\xi,\overline{Y}} [\overline{W}_k(k+1)] = \eta_{k+1}.
		\]
		
		So far, the lemma is proved.
	\end{proof}

	\noindent\textbf{Annealed coupling measure, several $\sigma$ fields and random variables:}
	
	From $P_E, P_{\overline{Y}}$ and $P_{\xi,\overline{Y}}$, there exist a probability measure $\overline{\P}$  such that for any positive and finite function $f$, we have
	\begin{equation}\label{eq61}
		\iiint f(\xi,\overline{Y},u) \overline{\P}(\d u,\d \xi,\d \overline{Y})
		=
		\iiint f(\xi,\overline{Y},u) P_{\xi,\overline{Y}}(\d u) P_E(\d \xi) P_{\overline{Y}} (\d \overline{Y}).
	\end{equation}
	And to avoid using too much notation, also, abuse use of notations, we use $ P_E\otimes P_{\overline{Y}} $ to represent the annealed coupling measure $\overline{\P}$.
	
	Under the annealed coupling measure $P_E\otimes P_{\overline{Y}}$, 
	for any integer $x\geqslant 0$, define
	\begin{equation}\label{def of Hx}
		\mathcal{H}_x:= \sigma\left(\{ \mathcal{G}_x,\{ F_j:j > -T_x \},\{\mathcal{T}_j:j>-T_x\} \}\right), \quad \mathcal{H}_\infty := \sigma\{\mathcal{H}_x:x\geqslant 0\}.
	\end{equation} 
	For any $m\geqslant 1$, denote by
	\begin{equation}\label{def of zeta}\zeta_m:=(F_{1-T_m},\cdots,F_{-T_{m-1}},\overline{Y}_{1-T_m}+m-1,\cdots,\overline{Y}_{-T_{m-1}}+m-1)
	\end{equation}
	the $m$-th excursion of $\overline{Y}$ and the environments on this interval.
	According to the independence between $\overline{Y}$ and $\xi$ under product measure $P_E\otimes P_{\overline{Y}}$, and by the construction about $\overline{Y}$, it is not hard to find that
	$\zeta_m$ are i.i.d random elements. Moreover, $\zeta_m$ is $\mathcal{H}_m$ measurable and independent with $\mathcal{H}_{m-1}$.
	
	\
	
	Following, we construct several random variables.
	For any $k\leqslant0$, we define
	\begin{equation}\label{def of Jk}
		J_k := \sum_{i=k+1}^{0} X_i -k\Lambda_s(\lambda_1)+\lambda \overline{Y}_{k}= -S_k-k\Lambda_s(\lambda)+\lambda \overline{Y}_k.
	\end{equation}
	And for any $x\geqslant1$, define
	\begin{equation}\label{def of R_x}
		R_x:= J_{-T_x} =  \sum_{i=-T_x+1}^{0} X_i + T_x\Lambda_s(\lambda)-\lambda x,
	\end{equation}
	with $R_0:=0$.
	Note that, $ R_x-R_{x-1} $ is a function about excursion $\zeta_{x}$. Therefore, under $P_E\otimes P_{\overline{Y}}$,
	$\{R_x:x\geqslant 0\}$ is a random walk respectively to $\mathcal{H}_x$ (i.e. for any $x\geqslant0$, $R_x$ is $\mathcal{H}_x$ measurable, and $R_{x+1}-R_x$ is independent with $\mathcal{H}_x$).
	
	Also, under $P_E\otimes P_{\overline{Y}}$, from the whole branching system $\mathcal{T}$, we define
	\begin{equation}\label{def of overlineA}
		B_x:= \left(1+\sum_{k=-T_x}^{-1}e^{J_x}\Phi(\mathcal{T}_k,\overline{Y}_k,0)\right)^{-1},
	\end{equation}
	where $T_x$ is defined in \cref{def of Tx} and has the property in \cref{27}, $\Phi(\mathcal{T}_k,\overline{Y}_k,0)$ is defined in \cref{def of Phik}, and $J_x$ is defined in \cref{def of Jk}.
	
	\begin{lemma}\label{Couping Lemma}(Coupling Lemma)
		For any integer $x\geqslant 1$,
		the law of $(B_x,R_x)$ under $P_E\otimes \tilde{P}_\xi$ is the same as the law of $(B_x,R_x)$ under $P_E\otimes P_{\overline{Y}}$.
		And in particular, we have
		\begin{equation}
			P_E\otimes P_{\xi}(M\geqslant x) = 
			P_E\otimes \tilde{P}_{\xi} (e^{R_x}B_x)
			=
			P_E\otimes P_{\overline{Y}}\left( e^{R_x} B_x \right).
		\end{equation}
	\end{lemma}
	
	\begin{proof}[Proof of \Cref{Couping Lemma}]
		By \Cref{Tanaka's construction}, we know that $\left(V(\omega_0)-x,V(\omega_1)-x,\cdots,V(\omega_{\tau_x})-x\right)$ has the same distribution as $ ( \overline{Y}_{-T_x},\overline{Y}_{1-T_x},\cdots,\overline{Y}_0) $. Then for any admissible $k$, and a measurable set $C$ of $\mathcal{P}(\N_0)^{\mathbb{Z}}$, and some measurable sets $A_{-k},A_{1-k},\cdots,A_0$ of $\R$, 
		and by the independence between $\xi$ and the trajectory of spine particles $\left(V(\omega_0),V(\omega_1),\cdots,V(\omega_{\tau_x})\right)$, we have
		\begin{align}
			&P_E\otimes \tilde{P}_\xi (T_k\xi \in C,V(\omega_0)-x \in A_{-k},V(\omega_1)-x\in A_{1-k},\cdots,V(\omega_{\tau_x})-x\in A_0|\tau_x=k)
			\nonumber\\
			=&
			P_E(T_k\xi \in C)
			\times
			P_E\otimes \tilde{P}_\xi \left(
			V(\omega_1)-x\in A_{1-k},\cdots,V(\omega_{\tau_x})-x\in A_0|\tau_x=k
			\right)
			\nonumber\\
			=&
			P_E(T_k\xi \in C)
			\times
			P_{\overline{Y}}
			\left(
			\overline{Y}_{-T_x}\in A_{-k},\overline{Y}_{1-T_x}\in A_{1-k},\cdots,\overline{Y}_0\in A_0|T_x=k
			\right)
			\nonumber\\
			=&
			P_E\otimes
			P_{\overline{Y}}
			\left(T_k\xi \in C,
			\overline{Y}_{-T_x}\in A_{-k},\overline{Y}_{1-T_x}\in A_{1-k},\cdots,\overline{Y}_0\in A_0|T_x=k
			\right)
		\end{align}
		Also, the distribution of $\tau_x$ under $P_E\otimes P_\xi$ is equal to the distribution of $T_x$ under $P_E\otimes P_{\overline{Y}}$.
		Therefore, by the law of total probability, we can deduce that
		the law of $(T_{\tau_x}\xi,V(\omega_0)-x,V(\omega_1)-x,\cdots,V(\omega_{\tau_x})-x)$ under $P\otimes \tilde{P}_\xi$ is equal
		the law of $(\xi,\overline{Y}_{-T_x},\overline{Y}_{1-T_x},\cdots,\overline{Y}_0)$ under $P_E\otimes P_{\overline{Y}}$.
		Therefore, by \Cref{Lemma of spine decomposition} and \Cref{mean of Phi}, and according to the constructions of $B_x$ under $P_E\otimes P_{\overline{Y}}$ and $P_E\otimes \tilde{P}_{\xi}$, the lemma is proved.
	\end{proof}
	
	\section{Proof of Main Results}\label{Section4}
	
	\subsection{Proof of \Cref{Quenched theorem}}
	
	For associated random walk $\{S_0,S_1,\cdots\}$ of environment $\xi$ (see \cref{quenched random walk}), 
	for any $n\geqslant1$, define the event
	\begin{equation*}
		\Delta_n:=\{|S_n-na|>n^{\frac{2}{3}}\}.
	\end{equation*}
	Under \Cref{subcritical}, by the moderate deviation principle (see Theorem 3.7.1 in \cite{LDP1998}), we know there exists some positive number $\theta>0$ such that
	\begin{equation*}
		\lim\limits_{n\rightarrow\infty} n^{-\frac{1}{3}}\log P_E(\Delta_n) = -\theta.
	\end{equation*}
	Thus, we have
	\begin{equation*}
		\sum_{n=1}^{\infty} P_E(\Delta_n)<\infty.
	\end{equation*}
	Therefore, by Borel-Cantelli Lemma, we have
	\begin{equation*}
		P_E\left(
		\bigcap\limits_{m=1}^\infty\bigcup\limits_{n=m}^\infty \Delta_n
		\right) = 0,
	\end{equation*}
	which implies that for $P$ almost surely environment $\xi$, there exists some $N(\xi)$ (depending on $\xi$), such that
	\begin{equation}\label{addition 1}
		|S_n-na|<n^{\frac{2}{3}}, \quad \forall n>N(\xi).
	\end{equation}
	
	In \cref{quenched express}, take $\lambda = \lambda_q = \kappa(-a)$, i.e. $\Lambda_s(\lambda_q) = -a$, then we have $R_x:= S_{\tau_x} -a\tau_x -\lambda_q x$.
	According to the fact $ \tau_x\geqslant x$ (under \cref{Jump Assumption}) and $B_x\leqslant 1$, for $x>N(\xi)$ large enough, by \cref{addition 1}, we have
	\begin{equation}\label{addition 2}
		e^{\lambda_q x} P_\xi(M\geqslant x)
		=
		e^{\lambda_q x} \tilde{P}_\xi\left[
		e^{R_x} B_x
		\right]
		\leqslant \tilde{P}_\xi 
		\left[
		\exp\left\{\tau_x^{\frac{2}{3}}\right\}
		\right].
	\end{equation}
	About the right side of above inequality, for any $\omega>0$, we have
	\begin{align}
		\tilde{P}_\xi 
		\left[
		\exp\left\{\tau_x^{\frac{2}{3}}\right\}
		\right]&\leqslant\tilde{P}_\xi 
		\left[
		\exp\left\{\omega \tau_x\right\}\right]
		+
		\tilde{P}_\xi\left[\exp\left\{\tau_x^{\frac{2}{3}}\right\};\tau_x^{\frac{2}{3}}\geqslant \omega \tau_x\right]
		\nonumber\\
		&= \tilde{P}_\xi 
		\left[
		\exp\left\{\omega \tau_x\right\}\right] + \tilde{P}_\xi\left[\exp\left\{\tau_x^{\frac{2}{3}}\right\};\tau_x\leqslant \omega^{-3}\right]
		\nonumber\\
		&\leqslant
		\tilde{P}_\xi 
		\left[
		\exp\{\omega \tau_x\}\right] + e^{\omega^{-2}}.
		\label{addition 3}
	\end{align}
	And about $\tilde{P}_\xi 
	\left[
	\exp\{\omega \tau_x\}\right]$, under \Cref{Jump Assumption}, by the strong Markov property of a random walk, we know that under $\tilde{P}_\xi$,  $\{\tau_s-\tau_{s-1};s\geqslant 1\}$ is a sequence of i.i.d random variables. Then we have
	\begin{align*}
		\tilde{P}_\xi 
		\left[
		\exp\{\omega \tau_x\}\right] 
		=
		\tilde{P}_\xi 
		\left[ 
		\exp\left\{\omega \sum\limits_{s=1}^x (\tau_s-\tau_{s-1})\right\}\right]
		=
		\left(\tilde{P}_\xi 
		\left[ 
		\exp\left\{\omega \tau_1\right\}\right]\right)^x.
	\end{align*}
	Therefore, 
	\begin{equation*}
		\frac{1}{x}\log \tilde{P}_\xi 
		\left[
		\exp\left\{\omega \tau_x\right\}\right] = \log \tilde{P}_\xi 
		\left[ 
		\exp\left\{\omega \tau_1\right\}\right].
	\end{equation*}
	Note that, under $\tilde{P}_\xi$, the random walk $\left( V(\omega_0),V(\omega_1),\cdots \right)$ has the same distribution as the random walk $(H_0,H_1,\cdots)$ under $ Q_\lambda $ which is defined in \Cref{Section2.1}.
	Therefore, we know that for some $\omega>0$ small enough such that
	\begin{equation*}
		\tilde{P}_\xi 
		\left[ 
		\exp\left\{\omega \tau_1\right\}\right]<\infty.
	\end{equation*}
	By dominated convergence theorem, we have
	\begin{equation*} 
		\lim\limits_{\omega\rightarrow 0^+}  \log \tilde{P}_\xi 
		\left[ 
		\exp\left\{\omega \tau_1\right\}\right] = 0,
	\end{equation*}
	which implies that
	\begin{equation} \label{addition 4}
		\lim\limits_{\omega\rightarrow 0^+} \lim\limits_{x\rightarrow\infty}\frac{1}{x}\log \tilde{P}_\xi 
		\left[
		\exp\{\omega \tau_x\}\right] = 0.
	\end{equation}
	By \cref{addition 4}, due to the arbitrariness of $\omega$ in \cref{addition 3}, we can deduce that
	\begin{equation}
		\lim\limits_{x\rightarrow\infty} \frac{1}{x}\log\tilde{P}_\xi 
		\left[
		\exp\left\{\tau_x^{\frac{2}{3}}\right\}
		\right] =0.
	\end{equation}
	Therefore, together with \cref{addition 2}, we have
	\begin{equation}\label{addition 10}
		\limsup\limits_{x\rightarrow\infty}\frac{1}{x}\log P_\xi(M\geqslant x)+\lambda_q \leqslant 0, \quad P_E~a.s..
	\end{equation}
	
	For lower bound, also, for any $x>N(\xi)$, we have
	\begin{align}\label{addition 9}
		e^{\lambda_q x} P_\xi (M\geqslant x)
		\geqslant
		P_\xi\left[ \exp\left\{-\tau_x^{\frac{2}{3}}\right\} B_x\right]
	\end{align}
	By Cauchy inequality, we have
	\begin{equation}\label{addition 7}
		P_\xi\left[ \exp\left\{-\tau_x^{\frac{2}{3}}\right\} B_x\right]
		P_\xi\left[B_x^{-1}\right] \geqslant
		P_\xi\left[ \exp\left\{-\frac{1}{2}\tau_x^{\frac{2}{3}}\right\} \right]^2
	\end{equation}
	About $P_\xi\left[ \exp\left\{-\frac{1}{2}\tau_x^{\frac{1}{3}}\right\} \right]$, by the same arguments in \cref{addition 3,addition 4}, it is not hard to verify that
	\begin{equation}\label{addition 8}
		\lim\limits_{x\rightarrow\infty} \frac{1}{x}\log P_\xi\left[ \exp\left\{-\frac{1}{2}\tau_x^{\frac{2}{3}}\right\} \right] = 0.
	\end{equation}
	About $P_\xi\left[B_x^{-1}\right]$, according to the construction on $B_x$, and by \Cref{mean of Phi}, we have
	\begin{align}
		P_\xi[B_x^{-1}|\mathcal{G}] &= 1+\sum_{k=0}^{\tau_x-1} \eta_{k+1}e^{\lambda(H_k-x)+(k-\tau_x)a-(S_k-S_{\tau_x})}
		\nonumber\\
		&\leqslant
		1+e^{S_{\tau_x}-a\tau_x}\sum_{k=0}^{\tau_x-1} \eta_{k+1}e^{ka-S_k}\nonumber\\
		&\leqslant
		1+\exp\left\{\tau_x^{\frac{2}{3}}\right\} \left(\sum_{k=0}^{N(\xi)} \eta_{k+1}e^{ka-S_k}+\sum_{k=N(\xi)+1}^{\tau_x-1} \eta_{k+1}e^{k^{\frac{1}{3}}}\right)
		\nonumber\\
		&\leqslant
		1+ M(\xi) \exp\{\tau_x^{\frac{2}{3}}\}
		+\exp\left\{\tau_x^{\frac{2}{3}}\right\} \left(\sum\limits_{k=1}^{\tau_x} \eta_{k}\right)\label{addition 5}
	\end{align}
	where $M(\xi):= \sum\limits_{k=0}^{N(\xi)} \eta_{k+1}e^{ka-S_k}<\infty$ depends only on the environment $\xi$.
	
	About $\sum_{k=1}^{n} \eta_{k}$, under the \cref{addition assumption for quenchde probability} in \Cref{Quenched theorem}, by the law of large number, we know that there exists some large $N_2(\xi)$, such that
	\[
	\sum_{k=1}^{n} \eta_{k} \leqslant n(\E\left[\eta_1\right]+1), \quad \forall ~n>N_2(\xi).
	\]
	Therefore, for $x$ larger enough ($x>N_2(\xi)$), with \cref{addition 5}, we have
	\begin{equation}
		P_\xi[B_x^{-1}] \leqslant
		1+ M(\xi)P_\xi\left[\exp\left\{\tau^{\frac{1}{3}}\right\}\right]
		+(\E\left[\eta_1\right]+1) P_\xi\left[
		\tau_x\exp\left\{\tau^{\frac{2}{3}}\right\}
		\right]
	\end{equation}
	About $P_\xi\left[
	\tau_x\exp\left\{\tau^{\frac{2}{3}}\right\}
	\right]$, for $x$ large enough, we have
	\begin{equation}
		P_\xi\left[\tau_x\exp\left\{\tau^{\frac{2}{3}}\right\} \right] \leqslant P_\xi\left[\exp\left\{\tau_x^{\frac{3}{4}}\right\}\right],
	\end{equation}
	and with the same arguments in \cref{addition 3,addition 4}, we have
	\begin{equation}
		0\leqslant\lim\limits_{x\rightarrow\infty} \frac{1}{x}\log P_\xi\left[\tau_x\exp\left\{\tau^{\frac{1}{3}} \right\}\right] \leqslant
		\lim\limits_{x\rightarrow\infty} \frac{1}{x}\log P_\xi\left[\exp\left\{\tau^{\frac{1}{2}} \right\}\right]=0.
	\end{equation}
	Therefore, combining with above relationships, we have
	\begin{equation}\label{addition 6}
		\lim\limits_{x\rightarrow\infty} \frac{1}{x}\log P_\xi[B_x^{-1}] = 0.
	\end{equation}
	With \cref{addition 6,addition 7,addition 8,addition 9}, we know
	\begin{equation}\label{addition 11}
		\lambda_q +\liminf\limits_{x\rightarrow\infty} \frac{1}{x} \log P_\xi\left(M\geqslant x\right) \geqslant 0, \quad P_E~a.s..
	\end{equation}
	Together with \cref{addition 10,addition 11}, \Cref{Quenched theorem} is proved.

	\subsection{Proof of \Cref{Theorem of strongly cases} and \Cref{Theorem of intermediately cases}}\label{Section3.2}

	For any $k\leqslant -1$ and $l\geqslant 1$, there exists a probability measure $Q_E^{(k,l)}$ on $ \mathcal{P}\left(\N_0\right)^{-k} \otimes \mathcal{P}\left(\N_0\right)^l $ such that for any measurable, bounded and positive function $f$, 
	\begin{equation*}
		Q_E^{(k,l)} (f(F_{k+1},F_{k+2},\cdots,F_0, F_1,\cdots,F_{l}))
		=
		P_E \left(
		f(F_{k+1},F_{k+2},\cdots,F_0, F_1,\cdots,F_{l})
		e^{\sum\limits_{i=k+1}^{0} X_i-k\Lambda_e(1)}
		\right)
	\end{equation*}
	In view of the exponential martingale property, it is not hard to verify the consistency of this family of measures. Hence, the Kolmogorov consistency theorem guarantees the existence of a measure $Q_E$ on $\mathcal{P}\left(\N_0\right)^{\mathbb{Z}}$ such that for any $k\leqslant0$ and $l\geqslant 1$, and any suitable function $f$,
	\begin{equation*}
		Q_E\left(
		f(F_{k+1},F_{k+2},\cdots,F_0,F_1,\cdots,F_l)
		\right) = P_E\left(
		f(F_{k+1},F_{k+2},\cdots,F_0,F_1,\cdots,F_l) e^{\sum\limits_{i=k+1}^{0} X_i-k\Lambda_e(1)}
		\right).
	\end{equation*}
	Subsequently, by a standard monotone class argument, we obtain that for any $k\leqslant 0$ and any suitable function $f$, 
	\begin{equation*}
		Q_E\left(
		f(F_{k+1},F_{k+2},\cdots,F_0,F_1,\cdots)
		\right) = P_E\left(
		f(F_{k+1},F_{k+2},\cdots,F_0,F_1,\cdots) e^{\sum\limits_{i=k+1}^{0} X_i-k\Lambda_e(1)}
		\right).
	\end{equation*}
	Note that, about the two-side environment $\xi$, for any positive index $k$, the marginal distribution of $F_k$ under $Q_E$ is the same as the marginal distribution of $F_k$ under $P_E$, while for any non-positive index $k$, the marginal distribution of $F_k$ is give by a exponent measure change (i.e. for any measurable function $\varphi$, we have $Q_E(\varphi(F_k)) = P_E\left( \varphi(F_k) e^{X_k-\Lambda_e(1)} \right)$). Moreover, no matter under $Q_E$ or $P_E$, all these random elements are independent with each other.
	By straightforward differentiation, for any non-positive index $k\leqslant 0$, we have
	\begin{align}\label{79}
		Q_E\left(X_{k}\right) = \Lambda_e^{\prime}(1).
	\end{align}
	
	Similar to \cref{eq61}, abuse use of notation, we also use $Q_E\otimes P_{\overline{Y}}$ to represent the probability measure such that for any positive and bounded function $f$, we have
	\begin{equation}\label{eq84}
		\iiint f(\xi,\overline{Y},u) Q_E\otimes P_{\overline{Y}}(\d u,\d \xi,\d \overline{Y})
		=
		\iiint f(\xi,\overline{Y},u) P_{\xi,\overline{Y}}(\d u) Q_E(\d \xi) P_{\overline{Y}} (\d \overline{Y}),
	\end{equation}
	where $ P_{\xi,\overline{Y}}(\d u) $ is the transfer measure defined in \Cref{Section2.3}, which is used to describe the subtrees attached to the spine.
	
	Recall that $\lambda_1$ is the unique positive number which satisfies
	\begin{equation}
		\Lambda_e(1)+\Lambda_s(\lambda_1)=0
	\end{equation}
	For any $k\leqslant0$, recall the definition of $J_k$ (see \cref{def of Jk})
	\begin{equation}\label{def of Jk2}
		J_k := \sum_{i=k+1}^{0} X_i -k\Lambda_s(\lambda_1)+\lambda_1 \overline{Y}_{k}= -S_k-k\Lambda_s(\lambda_1)+\lambda_1 \overline{Y}_k.
	\end{equation}
	And for any $x\geqslant 1$, recall the definition of $R_x$ (see \cref{def of R_x}),
	\begin{equation}\label{def of R_x2}
		R_x :=J_{-T_k}= \sum_{i=-T_x+1}^{0} X_i +T_x\Lambda_s(\lambda_1)-\lambda_1 x
		=
		-S_{-T_x} +T_x\Lambda_s(\lambda_1)-\lambda_1 x,
	\end{equation}
	with $R_0 = 0$, where $T_k$ is defined in \cref{def of Tx} and has the property in \cref{27}. 
	
	It is not hard to verify the independence of $\xi$ and $\overline{Y}$ (under $Q_E\otimes P_{\overline{Y}}$). By the independence and the construction of $\overline{Y}$, we know that $\{R_x,x\geqslant0\}$ is also a random walk under $Q_E\otimes P_{\overline{Y}}$, and the drift (or mean) of the random walk is given by
	\begin{align}
		Q_E\otimes P_{\overline{Y}} \left( R_1\right)
		&=
		P_{\overline{Y}}\otimes Q_E\left( \sum_{i=-T_1+1}^{0} X_i +T_1\Lambda_s(\lambda_1)-\lambda_1 \right)
		\nonumber\\
		&=
		P_{\overline{Y}} \left( T_1\Lambda_e^\prime(1) +T_1\Lambda_s(\lambda_1)-\lambda_1 \right)
		\nonumber\\
		&= \frac{\Lambda_e^\prime(1) +\Lambda_s(\lambda_1)-\lambda_1\Lambda_s^\prime(\lambda_1)}{\Lambda_s^\prime(\lambda_1)} = \frac{\varTheta(1)}{\Lambda_s^\prime(\lambda_1)}.\label{mean of R}
	\end{align}
	By \Cref{Def of classfy}, under $Q_E\otimes P_{\overline{Y}}$, for Class \Romannum{1} cases, the random walk has negative drift. While for Class \Romannum{2} cases, the random walk has zero drift.
	
	The relationship between $Q_E\otimes P_{\overline{Y}}$ and $P_E\otimes P_{\overline{Y}}$ can be described by following lemma.
	\begin{lemma}\label{Measure change Relationship 1}
		For any integer $x\geqslant 1$, for any measurable function $\phi$, we have
		\begin{equation}
			e^{\lambda_1 x}P_E\otimes P_{\overline{Y}} (e^{R_x}\phi(B_x,R_x)) = Q_E\otimes P_{\overline{Y}} ( \phi(B_x,R_x) ).
		\end{equation}
		where $B_x$ is defined in \cref{def of overlineA} from branching system $(\mathcal{T},V)$.
		Especially, we have
		\begin{equation}
			e^{\lambda_1 x} P_E\otimes P_\xi(M\geqslant x)
			=
			e^{\lambda_1 x} P_E\otimes P_{\overline{Y}}(e^{R_x}B_x)
			=
			Q_E\otimes P_{\overline{Y}} (B_x).
		\end{equation}
	\end{lemma}
	
	\begin{proof}[Proof of \Cref{Measure change Relationship 1}]
		For any $k\geqslant1$, we have
		\begin{align*}
			e^{\lambda_1 x}P_E\otimes P_{\overline{Y}} \left(e^{R_x}\phi(B_x,R_x),T_x=k\right)
			=
			P_E\otimes P_{\overline{Y}}\left( e^{\sum\limits_{i=-k+1}^{0} X_i +k\Lambda_s(\lambda_1)}\phi(B_x,R_x),T_x=k \right)
		\end{align*}		
		
		Thus, by the relationship between $P_E$ and $Q_E$, by Fubini's Theorem, we have
		\begin{align*}
			&P_E\otimes P_{\overline{Y}}\left( e^{\sum\limits_{i=-k+1}^{0} X_i +k\Lambda_s(\lambda_1)}\phi(B_x,R_x),T_x=k \right)
			\\=&
			\iiint e^{\sum\limits_{i=-k+1}^{0} X_i +k\Lambda_s(\lambda_1)}\phi(B_x,R_x) \ind{T_x =k} P_{\xi,\overline{Y}}(\d u) P_E(\d \xi) P_{\overline{Y}} (\d \overline{Y})
			\\=&
			\int\left(\iint \phi(B_x,R_x) \ind{T_x=k} P_{\xi,\overline{Y}}(\d u) P_{\overline{Y}} (\d \overline{Y}) \right) e^{\sum\limits_{i=-k+1}^{0} X_i +k\Lambda_s(\lambda_1)}  P_E(\d \xi) 
			\\=&
			\int\left(\iint \phi(B_x,R_x) \ind{T_x=k} P_{\xi,\overline{Y}}(\d u) P_{\overline{Y}} (\d \overline{Y}) \right) Q_E(\d \xi) 
			\\
			=&
			\iiint \phi(B_x,R_x) \ind{T_x=k}  P_{\xi,\overline{Y}}(\d u) Q_E(\d \xi) P_{\overline{Y}} (\d \overline{Y})
			\\
			=&
			Q_E\otimes P_{\overline{Y}} \left(
			\phi(B_x,R_x),T_x=k
			\right).
		\end{align*}
		Combining with above equation and summing over $k$, we have
		\begin{align}
			e^{\lambda_1 x}P_E\otimes P_{\overline{Y}} \left(e^{R_x}\phi(B_x,R_x)\right)
			&=
			\sum_{k=1}^{\infty}e^{\lambda_1 x}P_E\otimes P_{\overline{Y}} \left(e^{R_x}\phi(B_x,R_x),T_x=k\right)
			\nonumber\\
			&=\sum_{k=1}^{\infty}
			Q_E\otimes P_{\overline{Y}} \left( \phi(B_x,R_x),T_x=k \right)
			=
			Q_E\otimes P_{\overline{Y}} \left( \phi(B_x,R_x)\right).\nonumber
		\end{align}
		Combining with \Cref{Couping Lemma}, the proof is finished.
	\end{proof}
	
	\begin{remark}
		The \Cref{Measure change Relationship 1} can be translate by the Randon-Nikodym derivative. That is, on $\sigma$ filed $\mathcal{H}_x$ (see \cref{def of Hx}), we have
		\begin{equation}
			\left.\frac{\d Q_E\otimes P_{\overline{Y}}}{\d P_E\otimes P_{\overline{Y}}} \right|_{\mathcal{H}_x}
			= e^{R_x+\lambda_1x}.
		\end{equation}
	\end{remark}

	\
	
	Firstly, we consider the Class \Romannum{1} cases, and provide the following lemma.
	\begin{lemma}\label{Lemma of Strongly}
		Under the Assumptions of \Cref{Theorem of strongly cases}, under $Q_E\otimes P_{\overline{Y}}$, $\{R_x,x\geqslant0\}$ is a random walk with negative drift.
		And for $Q_E\otimes P_{\overline{Y}}$ almost sure $(\xi,\overline{\textbf{Y}})$, we have
		\begin{equation*} \sum_{k=-\infty}^{-1}e^{J_x}\Phi(\mathcal{T}_k,\overline{Y}_k,0) <\infty, \quad P_{\xi,\overline{Y}} \quad  a.s..
		\end{equation*}
		Therefore, according to the definition of $B_x$ (see \cref{def of overlineA}), we have
		\begin{equation*}
			\lim\limits_{x\rightarrow\infty}B_x =B_\infty \in (0,1],\quad Q_E\otimes P_{\overline{Y}}\quad a.s..
		\end{equation*}
	\end{lemma}
	
	\begin{proof}[Proof of \Cref{Lemma of Strongly}]
		Under the Assumptions of \Cref{Theorem of strongly cases}, and by \cref{mean of R}, we know that $\{R_x,x\geqslant0\}$ is a random walk with negative drift.
		And then, there exists some $\delta>0$ small enough, such that 
		\begin{equation*}
			\Lambda_e^\prime(1)+\Lambda_s(\lambda_1)-\lambda_1\Lambda_s^\prime(\lambda_1)<-\delta<0, \quad \text{and} \quad 
			\delta<\lambda_1\Lambda_s^{\prime}(\lambda_1).
		\end{equation*}
		Note that, by \cref{79}, under $Q_E\otimes P_{\overline{Y}}$, $\left\{-S_{k} = \sum\limits_{i=-k+1}^{0} X_i,k\leqslant 0\right\}$ is a random walk with mean $\Lambda_e^\prime(1)$. Therefore, by the law of large number, there exists some finite and negative integer $K_1$ (may depend on $\xi$), such that for any $k<K_1<0$,
		\begin{equation}\label{58}
			-S_{k}-k\Lambda_s(\lambda_1) \leqslant
			-k\Lambda_e^\prime(1)-k \Lambda_s(\lambda_1)
			\leqslant -k(\lambda_1\Lambda_s^\prime(\lambda_1) -\delta).
		\end{equation}
		By \cref{addition assumption for strongly cases} of \Cref{Theorem of strongly cases}, under $Q_E\otimes P_{\overline{Y}}$, for any $k\leqslant 0$, $\{\eta_{k}:k\leqslant 0\}$ are i.i.d random variables, and satisfy
		\begin{equation*}
			Q_E\otimes P_{\overline{Y}} \left( \log_+\eta_{k} \right)<\infty.
		\end{equation*}
		So, by a simple application of Fubini's theorem, for any $\delta>0$, we have
		\begin{equation*}
			\sum_{k=1}^{\infty}Q_E\otimes P_{\overline{Y}} \left( \log_+\eta_{k}>\frac{\delta}{2}k \right) <\infty
		\end{equation*}
		Then, by Borel-Cantelli Lemma, we deduce that
		\begin{equation*}
			Q_E\otimes P_{\overline{Y}} \left(\bigcup\limits_{m\leqslant 0}\bigcap\limits_{k\leqslant m} \{ \eta_{k} <e^{-\frac{\delta}{2}k}\}\right) =1,
		\end{equation*}
		which implies that there exists some negative integer $K_2$ (also may depend on $\xi$), such that for any $k<K_2$, we have
		\begin{align}\label{59}
			\eta_{k}<e^{-\frac{\delta}{2}k}.
		\end{align}
		
		Combining with \cref{58,59}, by \cref{def of Jk}, and according to \cref{27} (if $k<-T_x$, we have $\overline{Y}_{k}\leqslant-x$), we have
		\begin{align}
			\sum_{k=-\infty}^{-1} \eta_{k+1} e^{J_k}
			&=
			\sum_{k=-\infty}^{-1} \eta_{k+1} e^{-S_k-k\Lambda_s(\lambda_1)+\lambda_1\overline{Y}_k}
			\nonumber\\
			&\preccurlyeq
			\sum_{k=-\infty}^{-1} e^{-\frac{\delta}{2}k} e^{\delta k-k\lambda_1\Lambda^\prime_s(\lambda_1)+\lambda_1\overline{Y}_k}
			\nonumber\\
			&=\sum_{x=0}^{\infty} 
			\sum_{k=-T_{x+1}}^{-T_{x}-1}  e^{\frac{\delta}{2} k-k\lambda_1\Lambda^\prime_s(\lambda_1)+\lambda_1\overline{Y}_k}
			\nonumber\\
			&\leqslant
			\sum_{x=0}^{\infty} 
			\sum_{k=-T_{x+1}}^{-T_{x}-1} e^{(\frac{\delta}{2}-\lambda_1\Lambda^\prime_s(\lambda_1))k} e^{-\lambda_1 x}
			\nonumber\\
			&\leqslant\sum_{x=0}^{\infty} (T_{x+1}-T_x)e^{-(\frac{\delta}{2}-\lambda_1\Lambda_s^\prime(\lambda_1))T_{x+1}-\lambda_1 x}
			\nonumber\\
			&\leqslant\sum_{x=0}^{\infty} T_{x+1} e^{-(\frac{\delta}{2}-\lambda_1\Lambda_s^\prime(\lambda_1))T_{x+1}-\lambda_1 x}.\label{73}
		\end{align}
		By \cref{mean of theta,def of Tx}, we know that under $Q_E\otimes P_{\overline{Y}}$, $\{T_x,x\geqslant 1\}$ is a random walk with positive drift $\Lambda_s^\prime(\lambda_1)^{-1}$. Again, by the law of large number,
		there exists some $N_2$ large enough, such that for any $x>N_2$, we have
		\begin{equation}
			T_x \leqslant \left(\Lambda_s^\prime(\lambda_1)^{-1}+ \frac{\delta}{2\lambda_1 \Lambda_s^\prime(\lambda_1)^{-2}}\right)x.
		\end{equation}
		Then, by \cref{73}, we have
		\begin{align}
			\sum_{k=-\infty}^{-1} \eta_{k+1} e^{J_k}
			&\preccurlyeq \sum_{x=0}^{\infty}  x \exp\left(
			\left(\lambda_1\Lambda_s^\prime(\lambda_1)-\frac{\delta}{2}\right)\left(
			\Lambda_s^\prime(\lambda_1)^{-1}+ \frac{\delta}{2\lambda_1 \Lambda_s^\prime(\lambda_1)^{-2}}
			\right)x-\lambda_1 x
			\right)
			\nonumber\\
			&\leqslant
			\sum_{x=0}^{\infty} x\exp\left(
			-\frac{\delta^2}{4\lambda_1 \Lambda_s^\prime(\lambda_1)^{-1}}x
			\right) <\infty.
		\end{align}
		which implies the series does converge.
		
		Given $ (\xi,\overline{\textbf{Y}}) $, under $ P_{\xi,\overline{Y}} $, by \Cref{mean of Phi}, for any $x\geqslant1$, for $Q_E\otimes P_{\overline{Y}}$ almost sure $(\xi,\overline{\textbf{Y}})$, we have
		\begin{align*}
			P_{\xi,\overline{Y}}\left(\sum_{k=-T_x}^{-1}e^{J_x}\Phi(\mathcal{T}_k,\overline{Y}_k,0)\right) &= \sum_{k=-T_x}^{-1} e^{J_k}P_{\xi,\overline{Y}} \left[ \Phi(\mathcal{T}_k,\overline{Y}_k,0) \right]
			&=\sum_{k=-T_x}^{-1} \eta_{k+1}e^{J_k}\leqslant\sum_{k=-\infty}^{-1} \eta_{k+1} e^{J_k}<\infty.
		\end{align*}
		And naturally, $ \{e^{J_x}\Phi(\mathcal{T}_k,\overline{Y}_k,0);k\leqslant1\} $ are  non-negative random variables. Therefore, we have
		\begin{equation*}
			\sum_{k=-\infty}^{-1}e^{J_x}\Phi(\mathcal{T}_k,\overline{Y}_k,0) <\infty, \quad P_{\xi,\overline{Y}} \quad a.s..
		\end{equation*}
		And the remaining part is a trivial conclusion.
		So far, the lemma is proved.
	\end{proof}
	
	Now, we are ready to prove \Cref{Theorem of strongly cases}.
	
	\begin{proof}[Proof of \Cref{Theorem of strongly cases}]
		According to \Cref{Couping Lemma,Measure change Relationship 1,Lemma of Strongly}, by dominated convergence theorem ($B_x\leqslant 1$), we have
		\begin{equation*}
			\lim\limits_{x\rightarrow\infty}e^{\lambda_1 x} \P(M\geqslant x)
			=
			\lim\limits_{x\rightarrow\infty} e^{\lambda_1 x}P_E\otimes P_{\overline{Y}} (e^{R_x}B_x)
			=
			\lim\limits_{x\rightarrow\infty} Q_E\otimes P_{\overline{Y}} (B_x)
			=
			Q_E\otimes P_{\overline{Y}} (B_\infty) \in (0,1].
		\end{equation*}
		Thus, \Cref{Theorem of strongly cases} is proved.
	\end{proof}
	\
	
	Next, we focus on the Class \Romannum{2} cases.
	
	Under the Assumptions of \Cref{Theorem of intermediately cases}, by \cref{mean of R}, we know that under $Q_E\otimes P_{\overline{Y}}$, $\{R_x:x\geqslant 1\}$ is a recurrent random walk (its mean is $0$).
	Also, under Assumptions \ref{subcritical} and \ref{Jump Assumption}, the variance of the random walk satisfies
	\begin{align}
		Q_E\otimes P_{\overline{Y}} \left(R_1^2\right)
		&=
		P_{\overline{Y}}\otimes Q_E\left(R_1^2\right)
		\nonumber\\
		&=
		P_{\overline{Y}} \left(
		T_1 \text{Var}_{Q_E}(X_0)+ \lambda_1^2(T_1\Lambda_s^{\prime}(\lambda_1)-1)^2
		\right)
		\nonumber\\
		&=
		P_{\overline{Y}} (T_1)\text{Var}_{Q_E}(X_0)
		+\lambda_1^2  \text{Var}_{P_{\overline{Y}}} \left( \Lambda_s^{\prime}(\lambda_1)T_1\right)<\infty.
		\label{finite variance}
	\end{align}
	By \cref{finite variance},
	it is well know that there exists some finite and positive constant $C$ (also see the Chapter 4 in \cite{B2}), such that 
	\begin{align}\label{sqrtx probability decay}
		\lim\limits_{x\rightarrow\infty} \sqrt{x}Q_E\otimes P_{\overline{Y}} \left( \max\limits_{0\leqslant j\leqslant x} R_j\leqslant 0 \right) = C.
	\end{align}
	And indeed, there exists some finite constant $C_0$, such that for any $x\geqslant 0$, we have
	\begin{equation}\label{60}
		\sqrt{x}Q_E\otimes P_{\overline{Y}} \left( \max\limits_{0\leqslant j\leqslant x} R_j\leqslant 0 \right) \leqslant C_0.
	\end{equation}
	
	Define $\sigma_0=0$, and for $s\geqslant0$, let
	$\sigma_{s+1}: = \inf\{k\geqslant \sigma_s+1:R_k > R_{\sigma_s}\}$.
	Then $\{\sigma_k:k\geqslant0\}$ is the strictly ascending ladder epochs of the random walk.
	Define the renewal function
	\begin{equation}
		V(x) := I\{x\leqslant 0\} +\sum_{s=1}^{\infty} Q_E\otimes P_{\overline{Y}} \left(
		R_{\sigma_s}\leqslant -x
		\right), \quad  x\in \R.
	\end{equation}
	By \cref{finite variance}, we know that under $Q_E\otimes P_{\overline{Y}}$,
	\begin{equation*}
		\left\{V(R_x) \ind{\max\limits_{0\leqslant i\leqslant x} R_i \leqslant 0},x\geqslant 0\right\} ~\text{is a non-negative martingale with mean $1$,}
	\end{equation*}
	respectively to the filtration $\{\mathcal{H}_x;x\geqslant 0\}$, where $\mathcal{H}_k$ is defined in \cref{def of Hx}.
	Thus, by the martingale property and Kolmogorov's extension theorem (i.e the Doob-h transform), there exists a probability measure $(Q_E\otimes P_{\overline{Y}})^{-}$ on $\mathcal{H}_\infty$, such that for any $x\geqslant 0$,
	\begin{equation}
		\d (Q_E\otimes P_{\overline{Y}})^{-} \big|_{\mathcal{H}_x} = V(R_x) \ind{\max\limits_{0\leqslant i\leqslant x} R_i \leqslant 0} \d Q_E\otimes P_{\overline{Y}} \big|_{\mathcal{H}_x}.
	\end{equation}
	Under the probability $ (Q_E\otimes P_{\overline{Y}})^{-} $, the Markov process $\{R_x:x\geqslant0\}$ is called the random walk conditional to stay non-positive.
	
	About the relationship between $Q_E\otimes P_{\overline{Y}}$ and $
	\left(Q_E\otimes P_{\overline{Y}}\right)^-$, we have following lemma:
	
	\begin{lemma}\label{lemma2 for intermediately}
		For a sequence uniformly bounded random variable $ \{\Pi_x:x\geqslant 0\} $, such that
		for any $x$, $\Pi_x$ is $\mathcal{H}_x$ measurable, and $\lim\limits_{x\rightarrow\infty}\Pi_x = \Pi_\infty,~ (Q_E\otimes P_{\overline{Y}})^- $ almost surely.
		Then, we have
		\begin{equation}
			\lim\limits_{x\rightarrow\infty}Q_E\otimes P_{\overline{Y}} \left( \Pi_x \big| \max\limits_{0\leqslant i\leqslant x} R_i \leqslant 0 \right)
			=
			(Q_E\otimes P_{\overline{Y}})^{-} ( \Pi_\infty ).
		\end{equation}
	\end{lemma}
	
	\begin{proof}[Proof of \Cref{lemma2 for intermediately}]
		See the proof of the Lemma 2.5 in \cite{Afanasyev2005} or the Chapter 5.2 in \cite{B2} for more details.
	\end{proof}
	
	Similarly to the \Cref{Lemma of Strongly}, we have following lemma for Class \Romannum{2} cases:
	
	\begin{lemma}\label{lemma1 for intermediately}
		Under the Assumptions of \Cref{Theorem of intermediately cases}, for $ (Q_E\otimes P_{\overline{Y}})^{-} $ almost surely $(\xi,\overline{\textbf{Y}})$, we have
		\begin{equation*} \sum_{k=-\infty}^{-1}e^{J_x}\Phi(\mathcal{T}_k,\overline{Y}_k,0) <\infty, \quad P_{\xi,\overline{Y}} \quad  a.s..
		\end{equation*}
		Therefore, we have
		\begin{equation*}
			\lim\limits_{x\rightarrow\infty}B_x =B_\infty \in (0,1],\quad \left(Q_E\otimes P_{\overline{Y}}\right)^{-}\quad a.s..
		\end{equation*}
	\end{lemma}
	
	\begin{proof}[Proof of \Cref{lemma1 for intermediately}]
		Firstly, We have
		\begin{align}
			\sum_{k=-\infty}^{-1} \eta_{k+1} e^{J_k}
			&=
			\sum_{x=0}^{\infty} 
			\sum_{k=-T_{x+1}}^{-T_x-1} \eta_{k+1} e^{J_k}
			\nonumber
			\\
			&=
			\sum_{x=0}^{\infty} e^{\sum\limits_{i=-T_x+1}^0 X_i +T_x\Lambda_s(\lambda_1)+\lambda_1 \overline{Y}_{-T_x}} 
			\left(\sum_{k=-T_{x+1}}^{-T_x-1} \eta_{k+1} e^{\sum\limits_{i=k+1}^{-T_x} X_i - (k+T_x)\Lambda_s(\lambda_1)+\lambda_1(\overline{Y}_{k}+x)} 
			\right)
			\nonumber\\
			&=\sum_{x=0}^{\infty}
			\chi(\zeta_{x+1})e^{R_x} \nonumber
		\end{align}
		where $\chi(\zeta_{x+1}):= \sum\limits_{k=-T_{x+1}}^{-T_x-1} \eta_{k+1} e^{\sum\limits_{i=k+1}^{-T_x} X_i - (k+T_x)\Lambda_s(\lambda_1)+\lambda_1(\overline{Y}_k+x)} $ is a measurable function about $\zeta_{x+1}$ (defined in \cref{def of zeta}).
		
		About $R_x$, similar to the proof of Lemma 2.7 in \cite{Afanasyev2005}, we know that for any $\epsilon>0$, we have
		\begin{equation}\label{Decay of R_x}
			e^{R_x} = O(e^{-x^{\frac{1}{2}-\frac{\epsilon}{2}}}), \quad (Q_E\otimes P_{\overline{Y}})^- \quad a.s..
		\end{equation}
		
		Next, we want to bound the probability$ \left(Q_E\otimes P_{\overline{Y}}\right)^- (\chi(\zeta_{x+1}) > y).$
		
		Note that, for any $x\geqslant0$, under $ Q_E\otimes P_{\overline{Y}} $, $\chi(\zeta_{x+1})$ is a measurable function about $\zeta_{x+1}$, thus is independent of $\mathcal{H}_x$ and measurable respectively with $\mathcal{H}_{x+1}$, and has the identical distribution with $\chi(\zeta_{1})$.
		Because $V(x)$ is a renewal function, the inequality
		\begin{equation*}
			V(x+y) \leqslant V(x)+V(y)
		\end{equation*}
		is valid.
		And according to the Doob-h transform of $\left(Q_E\otimes P_{\overline{Y}}\right)^-$, we have
		\begin{align}
			\left(Q_E\otimes P_{\overline{Y}}\right)^- (\chi(\zeta_{x+1}) > y)
			&=
			Q_E\otimes P_{\overline{Y}} \left(\ind{\chi(\zeta_{x+1}) > y} V(R_{x+1}) \ind{ \max\limits_{0\leqslant i\leqslant x+1} R_i \leqslant 0} \right)
			\nonumber\\
			&\leqslant
			Q_E\otimes P_{\overline{Y}} \left(\ind{\chi(\zeta_{x+1}) > y} \left(V(R_{x})+V(R_{x+1}-R_x)\right) \ind{ \max\limits_{0\leqslant i\leqslant x} R_i \leqslant 0} \right)
			\nonumber\\
			&\leqslant
			Q_E\otimes P_{\overline{Y}} \left(\chi(\zeta_{x+1})>y\right)
			\cdot Q_E\otimes P_{\overline{Y}} \left( V(R_x)  \ind{ \max\limits_{0\leqslant i\leqslant x} R_i \leqslant 0} \right)
			\nonumber\\
			&~~~+
			Q_E \otimes P_{\overline{Y}} \left(
			\ind{\chi(\zeta_{x+1}) > y} V(R_{x+1}-R_x) \right)
			\cdot  Q_E\otimes P_{\overline{Y}}
			\left( \max\limits_{0\leqslant i\leqslant x} R_i \leqslant 0
			\right)
			\nonumber\\
			:&= I_1(y) + I_2(y)\label{100}.
		\end{align}
		
		By \cref{additional assumption for intermediately} ($\delta$ is the constant in \cref{additional assumption for intermediately}) and the construction of $Q_E$, using inequality $
		\left(\sum\limits_{k=1}^{n} a_k\right)^{2+\delta}
		\leqslant n^{1+\delta} \sum\limits_{k=1}^{n} a_k^{2+\delta} $ (by H\"older inequality), we have
		\begin{align*}
			Q_E\otimes P_{\overline{Y}}\left(
			\left(
			\log_+ \max\limits_{-T_1\leqslant i< 0} \eta_{k+1}
			\right)^{2+\delta}\right)
			&\leqslant
			Q_E\otimes P_{\overline{Y}}
			\left(
			T_1^{1+\delta} \sum_{i=-T_x}^{-1} \left(\log_+\eta_{k+1}\right)^{2+\delta}
			\right)
			\\
			&= P_{\overline{Y}} (T_1^{2+\delta}) \cdot \E[ e^{X_1-\Lambda_e(1)}\left(\log_+\eta_1\right)^{2+\delta} ] <\infty.
		\end{align*}
		Also, we have
		\begin{align*}
			Q_E\otimes P_{\overline{Y}} \left(\left(
			\max\limits_{-T_1\leqslant k< 0}\left|\sum_{i=k+1}^{0} X_i\right|
			\right)^{2+\delta}\right) 
			&\leqslant
			Q_E\otimes P_{\overline{Y}}
			\left(  T_1^{1+\delta} \sum_{k=-T_1}^{-1}\left(
			|X_k|
			\right)^{2+\delta} \right)
			\\
			&=
			P_{\overline{Y}} (T_1^{2+\delta}) Q_E(|X_0|^{2+\delta})<\infty.
		\end{align*}
		And we have
		\begin{equation*}
			\chi(\zeta_1) \leqslant
			\left(\max\limits_{-T_1\leqslant k< 0} \eta_{k+1}\right)
			e^{ \max\limits_{-T_1\leqslant k< 0}\left|\sum\limits_{i=k+1}^{0} X_i\right| } e^{T_1 \Lambda_s(\lambda_1)}.
		\end{equation*}
		Then, by Markov inequality and H\"older inequality, for any $y>1$, we have
		\begin{align}
			I_1(y)=Q_E\otimes P_{\overline{Y}}
			\left(
			\chi(\zeta_1)>y\right)
			&\leqslant
			Q_E\otimes P_{\overline{Y}}
			\left(
			\log_+ \max\limits_{-T_1\leqslant i< 0} \eta_{k+1}
			+
			\max\limits_{-T_1\leqslant k< 0}\left|\sum_{i=k+1}^{0} X_i\right|
			+ T_1\Lambda_s(\lambda_1) > \log y
			\right)
			\nonumber\\
			&\leqslant
			\frac{3^{1+\delta}}{(\log y)^{2+\delta}}
			\bigg(Q_E\otimes P_{\overline{Y}}\left(
			\left(
			\log_+ \max\limits_{-T_1\leqslant i< 0} \eta_{k+1}
			\right)^{2+\delta}\right)
			\nonumber\\
			&~~~+
			Q_E\otimes P_{\overline{Y}} \left(\left(
			\max\limits_{-T_1\leqslant k< 0}\left|\sum_{i=k+1}^{0} X_i\right|
			\right)^{2+\delta}\right) 
			+
			\Lambda_s^{2+\delta}(\lambda_1)Q_E\otimes P_{\overline{Y}}
			\left(
			T_1^{2+\delta}
			\right)
			\bigg)
			\nonumber\\
			&\leqslant
			\frac{K_1}{(\log y)^{2+\delta}},\label{I1}
		\end{align}
		where $K_1$ is some finite constant.
		
		About $I_2(y)$, by the identically distribution property and Cauchy-Schwarz inequality, we have
		\begin{align}
			Q_E\otimes P_{\overline{Y}}
			\left(
			\ind{\chi(\zeta_{x+1}) > y} V(R_{x+1}-R_x)
			\right)
			&=
			Q_E\otimes P_{\overline{Y}}
			\left(
			\ind{\chi(\zeta_{x+1}) > y} V(R_{1})
			\right)
			\nonumber\\
			&\leqslant
			Q_E\otimes P_{\overline{Y}}
			\left(
			\chi(\zeta_{x+1}) > y\right)^{\frac{1}{2}}
			\cdot
			Q_E\otimes P_{\overline{Y}}
			\left( V(R_1) ^{2} \right)^{\frac{1}{2}}
			\nonumber\\
			&\leqslant
			K_1 (\log y)^{-1-\frac{\delta}{2}}
			\cdot K_2 =
			K_3 (\log y)^{-1-\frac{\delta}{2}},\label{I2}
		\end{align}
		where $K_2$ is some finite constant such that $  	Q_E\otimes P_{\overline{Y}}
		\left( V(R_1) ^{2} \right)^{\frac{1}{2}}<K_2 $, and $K_3 = K_1 K_2$.
		
		Then, together with \cref{sqrtx probability decay,100,I1,I2}, for $y>1$, and $x$ large enough, there exists some finite constant $K_4<\infty$, such that
		\begin{equation}
			(Q_E\otimes P_{\overline{Y}})^- (\chi(\zeta_{x+1})>y)
			\leqslant
			K_4\left(\frac{1}{(\log y)^{2+\delta}}
			+\frac{1}{(\log y)^{1+\frac{\delta}{2}}}\frac{1}{\sqrt{x}}\right).
		\end{equation}
		Especially, for $\varepsilon>0$, we take $y_x = e^{x^{\frac{1}{2}-\varepsilon}} $, we have
		\begin{align*}
			\sum_{x=1}^{\infty} 
			(Q_E\otimes P_{\overline{Y}})^- (\chi(\zeta_{x+1})> e^{x^{\frac{1}{2}-\varepsilon}}) 
			\leqslant
			K_4\sum_{x=1}^{\infty} 
			\left(
			x^{-(1+\frac{\delta}{2}-2\varepsilon-\varepsilon\delta)}
			+x^{-(1+\frac{\delta}{4}-\varepsilon-\frac{\delta\varepsilon}{2})}
			\right)
		\end{align*}
		We can take $\varepsilon$ small enough such that above series converges.
		Then, by Borel-Cantelli Lemma, we have
		\begin{equation}\label{Decay of Chi_x}
			\chi(\zeta_{x+1}) = O(e^{x^{\frac{1}{2}-\epsilon}}), \quad  \left(Q_E\otimes P_{\overline{Y}}\right)^-\quad a.s..
		\end{equation}
		Therefore, combining with \cref{Decay of R_x,Decay of Chi_x}, for $\varepsilon>0$ small enough, we have
		\begin{equation}
			\chi(\zeta_{x+1}) e^{R_x} = O\left(  e^{x^{\frac{1}{2}-\varepsilon} - x^{\frac{1}{2}-\frac{\varepsilon}{2}}}  \right),
		\end{equation}
		which implies
		\begin{equation}
			\sum_{k=-\infty}^{-1} \eta_{k+1}e^{J_k}=\sum_{x=0}^{\infty} \chi(\zeta_{x+1}) e^{R_x}<\infty, \quad (Q_E\otimes P_{\overline{Y}})^-\quad a.s..
		\end{equation}
		
		The remaining part of the proof is the same as that of \Cref{Lemma of Strongly}. And we omit these contents. 
	\end{proof}

	Differently from Class \Romannum{1} cases, we need an additional estimate to control the upper bound of the target probability for Class \Romannum{2} cases. To do this, we provide following lemma:
	
	\begin{lemma}\label{Modified Upper Estimate}
		Under the Assumptions of \Cref{Theorem of intermediately cases}, for any $p>1$, there exists some finite constant $C_p$, such that
		\begin{equation}
			Q_E\otimes P_{\overline{Y}} \left( e^{\frac{1}{p}R_x}B_x \right) \leqslant C_p<\infty.
		\end{equation}
	\end{lemma}
	
	\begin{proof}[Proof of \Cref{Modified Upper Estimate}]
		Take $q>1$ such that $\frac{1}{p}+\frac{1}{q}=1$.
		
		By \Cref{Couping Lemma,Measure change Relationship 1}, we have
		\begin{equation*}
			Q_E\otimes P_{\overline{Y}} \left( e^{\frac{1}{p}R_x}B_x \right)
			=
			e^{\lambda_1 x}
			P_E\otimes P_{\overline{Y}} \left(e^{(1+\frac{1}{p})R_x} B_x\right)
			=
			e^{\lambda_1 x}
			P\otimes \tilde{P}_\xi\left(e^{(1+\frac{1}{p})R_x}B_x\right).
		\end{equation*}
		
		For $P_E~a.s.$ environment $\xi$, 
		define $U:=\sup\{S_n-n\lambda_e(1);n\geqslant 0\}$.
		Note that, $\{S_n-n\lambda_e(1);n\geqslant 0\}$ is a random walk with log-Laplace exponent $ \overline{\Lambda}(\theta)=\Lambda_e(\theta)-\theta \Lambda_e(1) $, which has a positive zero point at $\theta=1$.
		Then, $P_E$ almost surely, we have $U<\infty$. 
		Moreover, we have
		\begin{equation*}
			P_E(U\geqslant y) \sim e^{-y}.
		\end{equation*}
		Note that, if $\xi\in \{U <\lambda x\}$, we have $R_x\leqslant 0,~ \tilde{P}_\xi~a.s.$. So, we have
		\begin{align*}
			\tilde{P}_\xi(e^{(1+\frac{1}{p})R_x}B_x)
			\leqslant\tilde{P}_\xi(e^{R_x}B_x),
		\end{align*}
		which implies that
		\begin{equation*}
			P\otimes \tilde{P}_\xi (e^{(1+\frac{1}{p})R_x}B_x;U<\lambda x)\leqslant
			P\otimes \tilde{P}_\xi (e^{R_x}B_x)\leqslant P\otimes \tilde{P}_\xi (e^{R_x})=e^{-\lambda_1 x}.
		\end{equation*}
		And if $\xi\in \{U\geqslant \lambda_1 x\}$, and the inequality $R_x\leqslant U-\lambda_1 x$ holds. Then we have
		\begin{equation*}
			\tilde{P}_\xi(e^{(1+\frac{1}{p})R_x}B_x)
			\leqslant e^{\frac{1}{p}(U-\lambda_1x)}
			\tilde{P}_\xi(e^{R_x}B_x)\leqslant e^{\frac{1}{p}(U-\lambda_1x)},
		\end{equation*}
		which implies that
		\begin{equation*}
			P_E\otimes \tilde{P}_\xi (e^{(1+\frac{1}{p})R_x}B_x;U\geqslant\lambda x)\leqslant
			P\left( e^{\frac{1}{p}(U-\lambda_1 x)};U\geqslant \lambda_1 x \right)
		\end{equation*}
		And according to the tail probability of $U$, it is not hard to verify that as $x\rightarrow\infty$, we have
		\begin{equation*}
			P\left( e^{\frac{1}{p}(U-\lambda_1 x)} |U\geqslant \lambda_1 x \right) \sim \int_{0}^{\infty} e^{\frac{1}{p}y} e^{-y} \d x <\infty.
		\end{equation*}
		Combining with above relationships, we know that there exists some finite constant $C_p$, such that
		\begin{equation*}
			P\otimes \tilde{P}_\xi\left(e^{\left(1+\frac{1}{p}\right)R_x}B_x\right)
			\leqslant C_p e^{-\lambda_1 x}.
		\end{equation*}
		The lemma is proved.
	\end{proof}
	
	Define $ \gamma_x = \inf\left\{0\leqslant k\leqslant x, R_k=\max\limits_{0\leqslant j\leqslant x} R_j\right\}$. 
	With \Cref{lemma1 for intermediately,lemma2 for intermediately,Modified Upper Estimate}, we have following lemma:
	
	\begin{lemma}\label{intermediately lemma}
		Under the Assumptions of \Cref{Theorem of intermediately cases}, for any $j\geqslant 0$, there exists some constant $ M_j\in (0,\infty)$ such that
		\begin{align}
			\lim\limits_{x\rightarrow\infty}\sqrt{x}Q_E\otimes P_{\overline{Y}}
			\left( B_x,\gamma_x = j
			\right) = M_j<\infty.
		\end{align}
		In addition, there exist some $1<q<\frac{3}{2}$ and finite constant $C_1$, such that \[M_j\leqslant C Q_E\otimes P_{\overline{Y}} \left(
		B_j,\gamma_j=j\right) \leqslant C_1 j^{-\frac{3}{2q}}. \]
		where $C$ is the same constant in \cref{sqrtx probability decay}.
	\end{lemma}
	
	\begin{proof}[Proof of \Cref{intermediately lemma}]
		Recall the $\sigma$ field $\mathcal{H}_x$ defined in \cref{def of Hx}, we have
		\begin{align}
			Q_E\otimes P_{\overline{Y}}
			\left( B_x,\gamma_x = j
			\right)
			&=
			Q_E\otimes P_{\overline{Y}}
			\left( B_x,\gamma_j = j, \max_{j \leqslant s\leqslant x} R_s-R_j \leqslant 0
			\right)
			\nonumber\\
			&=
			Q_E\otimes P_{\overline{Y}}\left(
			Q_E\otimes P_{\overline{Y}} \left( B_x,\max_{ j \leqslant s\leqslant x} R_s-R_j \leqslant 0 \bigg| \mathcal{H}_j
			\right),\gamma_j=j \right)\nonumber
		\end{align}
		And due to the event $\left\{\max\limits_{ j \leqslant s\leqslant x} R_s-R_j \leqslant 0\right\}$ is independent with $\mathcal{H}_j$, we have
		\begin{equation*}
			Q_E\otimes P_{\overline{Y}}
			\left( B_x, \max_{ j \leqslant s\leqslant x} R_s-R_j \leqslant 0\bigg| \mathcal{H}_j
			\right)
			=
			Q_E\otimes P_{\overline{Y}}
			\left( B_x\bigg| \max_{ j \leqslant s\leqslant x} R_s-R_j \leqslant 0, \mathcal{H}_j
			\right)
			Q_E\otimes P_{\overline{Y}} \left(
			\max_{0 \leqslant s\leqslant x-k} R_s \leqslant 0
			\right)
		\end{equation*}
		And by \Cref{lemma1 for intermediately,lemma2 for intermediately} (\Cref{lemma2 for intermediately} ensures that the conditions of \Cref{lemma1 for intermediately} are satisfied), we know that there exists some $\mathcal{G}_j$ measurable random variable $V_j$, such that
		\begin{align}
			\lim\limits_{x\rightarrow\infty}Q_E\otimes P_{\overline{Y}}
			\left( B_x \bigg| \max_{ j \leqslant s\leqslant x} R_s-R_j \leqslant 0, \mathcal{H}_j
			\right)=V_j.
		\end{align}
		And by \cref{sqrtx probability decay}, we have
		\begin{equation*}
			\lim\limits_{x\rightarrow\infty}\sqrt{x} Q_E\otimes P_{\overline{Y}} \left(
			\max_{0 \leqslant s\leqslant x-j} R_s \leqslant 0
			\right) = C.
		\end{equation*}
		Therefore, we can take $M_j = C Q_E\otimes P_{\overline{Y}}
		\left(
		V_j, \gamma_j = j
		\right)<\infty$.
		
		Next, we observe that $B_x\leqslant B_j$. Therefore, we have
		\begin{align*}
			M_j\leqslant
			CQ_E\otimes P_{\overline{Y}}
			\left(B_j,\gamma_j=j\right)
		\end{align*}
		Take $p>3$ large enough such that \Cref{Modified Upper Estimate} holds, and thus $q\in (1,\frac{3}{2})$ such that $\frac{1}{p}+\frac{1}{q} = 1$. 
		By \Cref{Modified Upper Estimate} and H\"older inequality, due to the fact $B_j\leqslant1$, we have
		\begin{align}
			Q_E\otimes P_{\overline{Y}}
			\left(B_j,\gamma_j=j\right)
			&\leqslant
			Q_E\otimes P_{\overline{Y}}
			\left(B_je^{\frac{1}{p^2}R_j} e^{-\frac{1}{p^2}R_j},\gamma_j=j\right)
			\nonumber\\
			&\leqslant
			Q_E\otimes P_{\overline{Y}}
			\left(B_j^{\frac{1}{p}}e^{\frac{1}{p^2}R_j} e^{-\frac{1}{p^2}R_j},\gamma_j=j\right)
			\nonumber\\
			&\leqslant
			Q_E\otimes P_{\overline{Y}}
			\left(
			e^{\frac{1}{p}R_j}B_j\right)^{\frac{1}{p}}
			\cdot Q_E\otimes P_{\overline{Y}}
			\left(
			e^{-\frac{q}{p^2}R_j},\gamma_j=j \right)^{\frac{1}{q}} \nonumber\\
			&\leqslant
			C_p^{\frac{1}{p}} Q_E\otimes P_{\overline{Y}}
			\left(
			e^{-\frac{q}{p^2}R_j},\gamma_j=j \right)^{\frac{1}{q}}
		\end{align}
		Under the Assumptions of \Cref{Theorem of intermediately cases}, by \cref{mean of R,finite variance}, we know that the random walk $\{R_x;x\geqslant 0\}$ has mean 0 and finite variance. Hence, it is well know (also, one can see the Chapter 4.4 in \cite{B2} for rigorous arguments) that there exists some constant $C_5$, such that as $j\rightarrow\infty$, we have
		\begin{equation*}
			Q_E\otimes P_{\overline{Y}}
			\left(
			e^{-\frac{q}{p^2}R_j},\gamma_j=j	\right) \sim  C_5 j^{-\frac{3}{2}}
		\end{equation*}
		Therefore, combining with above inequalities and relationships, we know there exists some constant $C_1$, such that we have
		\begin{equation*}
			M_j\leqslant CQ_E\otimes P_{\overline{Y}} (B_j,\gamma_j = j)
			\leqslant
			C_1j^{-\frac{3}{2q}}.
		\end{equation*}
		So far, the proof is finished.
	\end{proof}
	
	With above lemmas, we are ready to prove \Cref{Theorem of intermediately cases}.
	
	\begin{proof}[Proof of \Cref{Theorem of intermediately cases}]
		First, we claim that for any $\epsilon>0$, there exists $N(\epsilon)$ large enough, such that for any $x>N$,
		\begin{equation*}
			\sum_{j=N+1}^{x} Q_E\otimes P_{\overline{Y}}
			\left(
			B_x,\gamma_x = j
			\right)\leqslant \epsilon ~ Q_E\otimes P_{\overline{Y}}
			\left(
			\max_{0\leqslant j\leqslant x} R_j \leqslant 0
			\right)
		\end{equation*}
		In fact, by \cref{60}, for any $\delta\in (0,1)$, we have
		\begin{align*}
			\sum_{j=N+1}^{x} Q_E\otimes P_{\overline{Y}}
			\left(
			B_x,\gamma_x = j
			\right)
			&\leqslant
			\sum_{j=N+1}^{\abs{(1-\delta)x}}
			Q_E\otimes P_{\overline{Y}} \left(
			B_j,\gamma_j = j
			\right) Q_E\otimes P_{\overline{Y}}
			\left(
			\max_{j\leqslant s\leqslant x}R_s-R_j \leqslant 0
			\right)
			\nonumber\\
			&~~~+
			\sum_{j=\abs{(1-\delta)x}}^{x} Q_E\otimes P_{\overline{Y}} \left(
			B_j,\gamma_j = j
			\right) Q_E\otimes P_{\overline{Y}}
			\left(
			\max_{j\leqslant s\leqslant x}R_s-R_j \leqslant 0
			\right)
			\\
			&\leqslant
			Q_E\otimes P_{\overline{Y}}
			\left(
			\max_{0\leqslant s\leqslant \delta x}R_s \leqslant 0
			\right)
			\left(\sum_{j=N+1}^{\infty} Q_E\otimes P_{\overline{Y}} \left(
			B_j,\gamma_j = j
			\right)\right)
			\\
			&~~~+ C_1(\abs{(1-\delta)x})^{-\frac{3}{2q}} \sum_{j=0}^{\abs{\delta x+1}}  Q_E\otimes P_{\overline{Y}}
			\left(
			\max_{0\leqslant s\leqslant j}R_s \leqslant 0
			\right)
			\\
			&\leqslant
			\frac{C_0}{\sqrt{\delta}}
			\left(\sum_{j=N+1}^{\infty} Q_E\otimes P_{\overline{Y}} \left(
			B_j,\gamma_j = j
			\right)\right) x^{-\frac{1}{2}}
			+ C_0C_1 x^{\frac{1}{2}-\frac{3}{2q}} \frac{\sqrt{\delta}}{1-\delta}
		\end{align*}
		where $C_0$ is defined in \cref{60}, and $C_1$ is the constant in \Cref{intermediately lemma} .Therefore, for $q\in (1,\frac{3}{2})$, by \Cref{intermediately lemma}, we have $\sum\limits_{j=0}^{\infty} Q_E\otimes P_{\overline{Y}} \left(
		B_j,\gamma_j = j
		\right)<\infty$, then by choose suitable $\delta$ and $N$, such that
		\begin{equation*}
			\frac{C_0}{\sqrt{\delta}}\left(\sum_{j=N+1}^{\infty} Q_E\otimes P_{\overline{Y}} \left(
			B_j,\gamma_j = j
			\right)\right) +C_0C_1 \frac{\sqrt{\delta}}{1-\delta}<\epsilon.
		\end{equation*}
		Thus, the claim is proved.
		
		Go back to \Cref{Theorem of intermediately cases}, for any $\epsilon>0$, for $x>N(\epsilon)$ large enough and for any $K>N(\epsilon)$, we have
		\begin{align}
			\sqrt{x}Q_E\otimes P_{\overline{Y}} (B_x)
			&=
			\sum_{j=0}^{x} \sqrt{x}Q_E\otimes P_{\overline{Y}} (B_x,\gamma_x =j)
			\nonumber\\
			&=
			\sum_{j=0}^{K} \sqrt{x}Q_E\otimes P_{\overline{Y}} (B_x,\gamma_x = k)
			+
			\sum_{j=K+1}^{x} \sqrt{x}Q_E\otimes P_{\overline{Y}} (B_x,\gamma_x = j)
			\nonumber\\
			&:=I_1(x,K)+I_2(x,K).\nonumber
		\end{align}
		About $I_1(x,K)$, by \Cref{intermediately lemma}, we have
		\begin{equation*}
			\lim\limits_{x\rightarrow\infty}I_1(x,K) = \sum_{k=0}^{K} M_k<\infty.
		\end{equation*}
		About $I_2(x,K)$, by previous claim, we have
		\begin{equation*}
			I_2(x,K)\leqslant I_2(x,N(\epsilon)+1) \leqslant \epsilon \sqrt{x} Q_E\otimes P_{\overline{Y}}
			\left(
			\max_{0\leqslant j\leqslant x} R_j \leqslant 0
			\right)\leqslant \epsilon C_0.
		\end{equation*}
		where $C_0$ is defined in \cref{60}.
		Then, we have
		\begin{equation*}
			\liminf\limits_{x\rightarrow\infty}\sqrt{x}Q_E\otimes P_{\overline{Y}} (B_x)
			\geqslant
			\lim\limits_{x\rightarrow\infty} I_1(x) = 
			\sum_{j=0}^{K} M_j
		\end{equation*}
		On the other side, we have
		\begin{equation*}
			\limsup\limits_{x\rightarrow\infty}\sqrt{x}Q_E\otimes P_{\overline{Y}} (B_x)
			\leqslant
			\lim\limits_{x\rightarrow\infty} I_1(x,K) +I_2(x,K) = 
			\sum_{j=0}^{K} M_j +\epsilon C_3.
		\end{equation*}
		And due to the arbitrariness of $K$ and $\epsilon$, we have
		\begin{equation*}
			\lim\limits_{x\rightarrow\infty}\sqrt{x}Q_E\otimes P_{\overline{Y}} (B_x)
			= 
			\sum_{j=0}^{\infty} M_j<\infty
		\end{equation*}
		Together with \Cref{Couping Lemma,Measure change Relationship 1}, \Cref{Theorem of intermediately cases} is proved.
	\end{proof}

	\subsection{Proof of \Cref{Theorem of weakly cases}}
	
	In this subsection, denote by $\lambda_\rho$ the solution of equation $\Lambda_e(\rho)+\rho\Lambda_s(\lambda) =0$.
	
	With the same arguments in \Cref{Section3.2}, there exists a probability measure  $Q^\rho_E$ on $\mathcal{P}^{\mathbb{Z}}$ such that for $k\leqslant 0$, and any suitable measurable function $f$, we have
	\begin{equation*}
		Q^{\rho}_E\left(
		f(F_{k+1},F_{k+2},\cdots,F_0,F_1,\cdots)
		\right) = P_E\left(
		f(F_{k+1},F_{k+2},\cdots,F_0,F_1,\cdots) e^{\rho\sum\limits_{i=k+1}^{0} X_i-k\Lambda_e(\rho)}
		\right).
	\end{equation*}
	Note that, about the two-side environment $\xi$, for any positive index $k$, the marginal distribution of $F_k$ under $Q_E$ is the same as the marginal distribution of $F_k$ under $P_E$, while for any non-positive index $k$, the marginal distribution of $F_k$ is give by a exponent measure change (i.e. for any measurable function $\varphi$, we have $Q^\rho_E(\varphi(F_k)) = P_E\left( \varphi(F_k) e^{\rho X_k-\Lambda_e(\rho)} \right)$). Moreover, no matter under $Q^\rho_E$ or $P_E$, all these random elements are independent with each other.
	Also, by straightforward differentiation, for any non-positive index $k\leqslant 0$, we have
	\begin{align*}
		Q^\rho_E\left(X_{k}\right) = \Lambda_e^{\prime}(\rho).
	\end{align*}
	
	Again, similar to \cref{eq61,eq84}, abuse use of notation, we also use $Q_E^\rho\otimes P_{\overline{Y}}$ to represent the probability measure such that for any positive and bounded function $f$, we have
	\begin{equation}
		\iiint f(\xi,\overline{Y},u) Q^\rho_E\otimes P_{\overline{Y}}(\d u,\d \xi,\d \overline{Y})
		=
		\iiint f(\xi,\overline{Y},u) P_{\xi,\overline{Y}}(\d u) Q^\rho_E(\d \xi) P_{\overline{Y}} (\d \overline{Y}),
	\end{equation}
	where $ P_{\xi,\overline{Y}}(\d u) $ is the transfer measure defined in \Cref{Section2.3}.
	
	Under $Q^\rho_E\otimes P_{\overline{Y}}$, 
	for any $k\leqslant0$, recall the definition of $J_k$ (also see \cref{def of Jk})
	\begin{equation}\label{def of Jk3}
		J_k := \sum_{i=k+1}^{0} X_i -k\Lambda_s(\lambda_\rho)+\lambda_\rho \overline{Y}_{k}= -S_k-k\Lambda_s(\lambda_\rho)+\lambda_\rho \overline{Y}_k.
	\end{equation}
	And for any $x\geqslant 1$, recall the definition of $R_x$ (also see \cref{def of R_x}),
	\begin{equation}\label{def of R_x3}
		R_x :=J_{-T_k}= \sum_{i=-T_x+1}^{0} X_i +T_x\Lambda_s(\lambda_\rho)-\lambda_\rho x
		=
		-S_{-T_x} +T_x\Lambda_s(\lambda_\rho)-\lambda_\rho x,
	\end{equation}
	with $R_0 = 0$, where $T_k$ is defined in \cref{def of Tx} and has the property in \cref{27}. 
	By identical computations as \Cref{Section3.2}, it follows that under the Assumptions of \Cref{Theorem of weakly cases}, $\{R_x,x\geqslant0\}$ is also a random walk under $Q^\rho_E\otimes P_{\overline{Y}}$, and the drift (or mean) of the random walk is 0.

	The relationship between $Q^\rho_E\otimes P_{\overline{Y}}$ and $P_E\otimes P_{\overline{Y}}$ is described by following lemma.
	\begin{lemma}\label{Measure change Relationship 2}
		For any integer $x\geqslant 1$, for any measurable function $\phi$, we have
		\begin{equation}
			e^{\rho\lambda_\rho x}P_E\otimes P_{\overline{Y}} (e^{\rho R_x}\phi(B_x,R_x)) = Q^\rho_E\otimes P_{\overline{Y}} ( \phi(B_x,R_x) ),
		\end{equation}
		where $B_x$ is defined in \cref{def of overlineA} from branching system $(\mathcal{T},V)$.
		Especially, we have
		\begin{equation}
			e^{\rho\lambda_\rho x} P\otimes P_\xi(M\geqslant x)
			=
			e^{\rho\lambda_\rho x} P_E\otimes P_{\overline{Y}}(e^{R_x}B_x)
			=
			Q^\rho_E\otimes P_{\overline{Y}} (e^{(1-\rho)R_x}B_x).
		\end{equation}
	\end{lemma}

	\begin{proof}[Proof of \Cref{Measure change Relationship 2}]
		This proof differs nothing from that of the \Cref{Measure change Relationship 1}. And the results follow by verification in the same manner.
	\end{proof}
	
	Define $ \gamma_x = \inf\{0\leqslant k\leqslant x, R_k=\max\limits_{0\leqslant j\leqslant x} R_j\}$.
	
	\begin{lemma}\label{lemma of weakly 1}
		For any fixed integer $j\geqslant0$, we have
		\begin{equation}
			\liminf\limits_{x\rightarrow\infty} x^{\frac{3}{2}}
			Q_E^\rho\otimes P_{\overline{Y}}
			\left(e^{(1-\rho)R_x}B_x,\gamma_x=j\right) >0.
		\end{equation}
	\end{lemma}
	
	\begin{proof}[Proof of \Cref{lemma of weakly 1}]
		Recall the $\sigma$ field $\mathcal{H}_x,\mathcal{H}_\infty$ defined in \cref{def of Hx}.
		According to the construction about $B_x$, by Jensen's inequality and \Cref{mean of Phi}, we have
		\begin{equation}
			Q_E^{\rho}\otimes P_{\overline{Y}}
			[B_x|\mathcal{H}_\infty]
			\geqslant
			\left(
			1+\sum_{k=-T_x}^{-1} e^{J_k}\eta_{k+1}
			\right)^{-1}=:A_x,
		\end{equation}
		And $R_x$ is a $\mathcal{H}_x$ (also $\mathcal{H}_\infty$) measurable random variable. Thus, it is enough to prove that
		\begin{equation}
			\liminf\limits_{x\rightarrow\infty} x^{\frac{3}{2}}
			Q_E^\rho\otimes P_{\overline{Y}}
			\left(e^{(1-\rho)R_x}A_x,\gamma_x=j\right) >0.
		\end{equation}
		
		Recall $\zeta_{j+1}$ (defined in \cref{def of zeta}), we have
		\begin{align}
			A_x^{-1}
			&= 1+\sum_{k=-T_x}^{-1} e^{J_k}\eta_{k+1}
			\nonumber\\
			&= 1+\sum_{j=0}^{x-1} e^{\sum\limits_{i=-T_j+1}^0 X_i +T_x\Lambda_s(\lambda_1)+\lambda_1 \overline{Y}_{-T_j}} 
			\left(\sum_{k=-T_{j+1}}^{-T_j-1} \eta_{k+1} e^{\sum\limits_{i=k+1}^{-T_x} X_i - (k+T_j)\Lambda_s(\lambda_1)+\lambda(\overline{Y}_{k}+j-1)} 
			\right)
			\nonumber\\
			&=1+\sum_{j=0}^{x-1}
			\chi(\zeta_{j+1})e^{R_j} \nonumber
		\end{align}
		where $\chi(\zeta_{j}):= \sum\limits_{k=-T_{j+1}}^{-T_j-1} \eta_{k+1} e^{\sum\limits_{i=k+1}^{-T_j} X_i - (k+T_j)\Lambda_s(\lambda_1)}$ is a measurable function about excursion $\zeta_{j+1}$.
		Also, we have
		\begin{align}
			A_x^{-1}&= 1+\sum_{j=0}^{x-1}
			\chi(\zeta_{j+1})e^{R_j}\nonumber\\
			&=
			1+ e^{R_x} \sum_{j=0}^{x-1} \chi(\zeta_{j})e^{R_{j}-R_{j+1}}e^{R_{j+1}-R_x}
			\nonumber\\
			&=1+ e^{R_x} \sum_{j=0}^{x-1} \overline{\chi}(\zeta_{j}) e^{R_{j+1}-R_x}.
		\end{align}
		where $\overline{\chi}(\zeta_{x-j}):=\chi(\zeta_{x-j})e^{R_{j+1}-R_j}$ is a measurable function about $\zeta_{x-j}$.
		
		For any $\delta\in (0,1)$, define 
		\begin{equation*}
			G_x(\zeta_1,\zeta_2,\cdots,\zeta_{\delta x}):= 
			\sum_{j=0}^{\delta x-1} \chi(\zeta_{j+1}) e^{R_j}, \quad 
			\text{and}\quad  K_x(\zeta_1,\zeta_2,\cdots,\zeta_{\delta x})
			:= \sum_{j=1}^{\delta x} \overline{\chi}_j e^{-R_j}.
		\end{equation*}
		Repeating the proof of \Cref{lemma1 for intermediately}, with additional assumption \cref{additional assumption for weakly}, we have
		\begin{align*}
			\lim\limits_{x\rightarrow\infty}G_x = G_\infty<\infty,\quad (Q_E^{\rho}\otimes P_{\overline{Y}})^- \quad a.s..
			\nonumber\\
			\lim\limits_{x\rightarrow\infty}K_x = K_\infty<\infty,\quad (Q_E^{\rho}\otimes P_{\overline{Y}})^+ \quad a.s..
		\end{align*}
	
		Note that, for $\delta>\frac{1}{2}$, we have
		\begin{equation}
			A_x^{-1}
			\leqslant
			1+ G_x(\zeta_1,\zeta_1,\cdots,\zeta_{\delta x})
			+
			e^{R_x} K_x(\zeta_{x},\zeta_{x-1},\cdots,\zeta_{x+1-\delta x})=
			1+G_x+e^{R_x} \overline{K_x},
		\end{equation}
		Therefore, we have
		\begin{align}
			&Q_E^\rho\otimes P_{\overline{Y}}
			\left(e^{(1-\rho)R_x}A_x,\gamma_x=j\right)
			\nonumber\\
			\geqslant&
			Q_E^\rho\otimes P_{\overline{Y}}
			\left(e^{(1-\rho)R_x} \left(1+G_x+e^{R_n} \overline{K}_x \right)^{-1} ,\gamma_x=j\right)
			\nonumber\\
			= &Q_E^\rho\otimes P_{\overline{Y}}
			\left(e^{(1-\rho)R_x} \left(1+G_x+e^{R_n} \overline{K}_x\right)^{-1} ,\gamma_j=j,\max_{j\leqslant s\leqslant x} R_s-R_j\leqslant 0\right)\nonumber
		\end{align}
		By the Lemma 7.3 in \cite{B2} (i.e the Theorem 2.7 in \cite{Afanasyev2012}), we know that there exists some finite constant $C_j>0$, such that
		\begin{equation}
			\lim\limits_{x\rightarrow\infty} \dfrac{Q_E^\rho\otimes P_{\overline{Y}}
				\left(e^{(1-\rho)R_x} \left(A_x(m)+\overline{A}_x(m)\right)^{-1} ,\gamma_j=j,\max_{j\leqslant s\leqslant x} R_s-R_j\leqslant 0\right)}{Q_E^\rho\otimes P_{\overline{Y}}\left(
				e^{(1-\rho)R_x}, \gamma_j=j, \max_{j\leqslant s\leqslant x} R_s-R_j\leqslant 0
				\right) } =C_j.
		\end{equation}
		Also, by the Theorem 4.10 in \cite{B2} or the Proposition 2.1 in \cite{Afanasyev2012}, there exists some finite constant $D_j>0$, such that
		\begin{equation}
			\lim\limits_{x\rightarrow\infty}x^{\frac{3}{2}}Q_E^\rho\otimes P_{\overline{Y}}\left(
			e^{(1-\rho)R_x}, \gamma_j=j, \max_{j\leqslant s\leqslant x} R_s-R_j\leqslant 0
			\right) = D_j.
		\end{equation}
		Combining with above relationships, we have
		\begin{equation}
			\liminf\limits_{x\rightarrow\infty} x^{\frac{3}{2}}
			Q_E^\rho\otimes P_{\overline{Y}}
			\left(e^{(1-\rho)R_x}B_x,\gamma_x=j\right)\geqslant C_j D_j>0
		\end{equation}
		
		The proof is finished.
	\end{proof}

	About the upper bound of target probability, we have following lemma. Although it indeed provides control via an upper bound inequality, it is not sharp and thus insufficient to derive the exact polynomial decay rate of the annealed probability.
	\begin{lemma}\label{lemma of weakly 2}
		For any $\delta\in (0,\rho)$, there exists some constant $ C_\delta <\infty $, such that for any $x\geqslant1$, we have
		\begin{equation*}
			Q^\rho_E\otimes P_{\overline{Y}} \left(
			e^{(1-\delta)R_x} B_x
			\right) \leqslant C_\delta.
		\end{equation*}
	Then by the H\"older inequality, we have
	\begin{equation}
		Q^\rho_E\otimes P_{\overline{Y}} \left(
		e^{(1-\delta)R_x} B_x
		\right) \leqslant C_\delta^{\frac{1-\rho}{1-\delta}}<\infty.
	\end{equation}
\end{lemma}
	
	\begin{proof}[Proof of \Cref{lemma of weakly 2}]
	By \Cref{Measure change Relationship 2}, we have
	\begin{align}
		Q^\rho_E\otimes P_{\overline{Y}} \left(
		e^{(1-\delta)R_x} B_x
		\right) &= 
		e^{\rho\lambda_\rho x}P_E\otimes P_{\overline{Y}} \left(
		e^{(1+\rho-\delta) R_x} B_x
		\right)
		\nonumber\\
		&= e^{\rho\lambda_\rho x}
		P\otimes P_{\xi} \left( e^{(1+\rho-\delta) R_x} B_x \right).
	\end{align}
	
	Under $P\otimes \tilde{P}_\xi$, 
	for $P~a.s.$ environment $\xi$, 
	define $U:=\sup\{S_n+n\Lambda_s(\lambda_\rho);n\geqslant 0\}$.
	Note that, $\{S_n-n\Lambda_s(\lambda_\rho);n\geqslant 0\}$ is a random walk with log-Laplace exponent $ \overline{\Lambda}(\theta)=\Lambda_e(\theta)+\theta \Lambda_s(\lambda_\rho) $, which has a positive zero point at $\theta=\rho$ (under the Assumption of \Cref{Theorem of weakly cases}).
	Then, $P$ almost surely, we have $U<\infty$. 
	Moreover, we have
	\begin{equation*}
		P(U\geqslant y) \sim e^{-\rho y}.
	\end{equation*}
	Note that, if $\xi\in \{U <\lambda_\rho x\}$, we have $R_x\leqslant 0,~ \tilde{P}_\xi~a.s.$. So, we have
	\begin{align*}
		\tilde{P}_\xi(e^{(1+\rho-\delta)R_x}B_x)
		\leqslant\tilde{P}_\xi(e^{R_x}B_x),
	\end{align*}
	which implies that
	\begin{equation*}
		P\otimes \tilde{P}_\xi (e^{(1+\rho-\delta)R_x}B_x;U<\lambda x)\leqslant
		P\otimes \tilde{P}_\xi (e^{R_x}B_x).
	\end{equation*}
	
	And if $\xi\in \{U\geqslant \lambda_\rho x\}$, and the inequality $R_x\leqslant U-\lambda_1 x$ holds. Then we have
	\begin{equation*}
		\tilde{P}_\xi(e^{(1+\rho-\delta)R_x}B_x)
		\leqslant e^{(\rho-\delta)(U-\lambda_1x)}
		\tilde{P}_\xi(e^{R_x}B_x)\leqslant e^{\frac{1}{p}(U-\lambda_1x)},
	\end{equation*}
	which implies that
	\begin{equation*}
		P\otimes \tilde{P}_\xi (e^{(1+\frac{1}{p})R_x}B_x;U\geqslant\lambda x)\leqslant
		P\left( e^{\frac{1}{p}(U-\lambda_1 x)};U\geqslant \lambda_1 x \right)
	\end{equation*}
	And according to the tail probability of $U$, it is not hard to verify that as $x\rightarrow\infty$, we have
	\begin{equation*}
		P\left( e^{(\rho-\delta)(U-\lambda_1 x)}|U\geqslant \lambda_1 x \right) \sim \int_{0}^{\infty} e^{(\rho-\delta)y} e^{-\rho y} \d x <\infty.
	\end{equation*}
	Combining above relationships, there exists some finite constant $C$, such that
	\begin{equation*}
		Q_E^{\rho}\otimes P_{\overline{Y}} \left(
		e^{(1-\delta)R_x} B_x
		\right)\leqslant Q_E^{\rho} \left(
		e^{(1-\rho)R_x} B_x
		\right) +C.
	\end{equation*}
	And by H\"older inequality, we have
	\begin{equation*}
		Q_E^{\rho}\otimes P_{\overline{Y}} \left(
		e^{(1-\delta)R_x} B_x
		\right)
		\geqslant
		Q_E^{\rho}\otimes P_{\overline{Y}} \left(
		e^{(1-\rho)R_x} B_x
		\right) ^{\frac{1-\delta}{1-\rho}}
	\end{equation*}
    As long as $\delta \in (0,\rho)$, we know that the inequality $x^{\frac{1-\delta}{1-\rho}}\leqslant x+C$ implies that there exists some finite constant $C_\delta$ such that $x\leqslant C_\delta$. Therefore, we have
    \begin{equation*}
    	Q_E^{\rho}\otimes P_{\overline{Y}} \left(
    	e^{(1-\delta)R_x} B_x
    	\right) \leqslant C_\delta<\infty.
    \end{equation*}
    
    So far, the proof is finished.
    \end{proof}
	
	\begin{proof}[Proof of \Cref{Theorem of weakly cases}]
		With \Cref{Measure change Relationship 2,lemma of weakly 1,lemma of weakly 2}, \Cref{Theorem of weakly cases} is proved.
	\end{proof}
	
	\begin{remark}\label{R final}
		As mentioned earlier, for Class \Romannum{3} subcritical BRWRE, we fail to derive the exact polynomial decay rate due to the absence of crucial upper bound estimates. Specifically, we elaborate on this issue by drawing on relevant results concerning branching processes in random environments (BPRE). In the calculation of the extinction probability for BPRE, a robust upper bound is established such that $ P_\xi(Z_n>0)\leqslant \min \{e^{S_k};0\leqslant k\leqslant n\} = e^{L_n} $, where $L_n:=\min\{S_k;0\leqslant k\leqslant n\}$ is the minimal of the associated random walk $\{S_n,x\geqslant 0\}$, and this estimate is sharp, which is sufficient to accomplish the required upper bound control, see \cite{Geiger2003,B2} and references therein for detail discussions.
		
		However, for subcritical branching random walk in random environment, borrow some notations in \cref{Section2}, by the fact $\{M\geqslant x\}\subset \{M\geqslant y\}$ for any $y\leqslant x$, we can deduce that
		\[
		P_\xi\left(M\geqslant x\right) \leqslant 
		\min \{
		\tilde{P}_\xi \left(e^{R_y}\right),0\leqslant y\leqslant x
		\}.
		\]
		While, under the measure \(\tilde{P}_\xi\), we technically introduce additional information about the spine, which prevents the minimum operation in the above formula from being commuted into the expectation without any loss. And this is the intrinsic reason why we can't obtain the exactly polynomial decay.
	\end{remark}

	\bibliographystyle{plain}
	\bibliography{subcritical.bib}

@article {Geiger2003,
	AUTHOR = {Geiger, J. and Kersting, G. and Vatutin, V. A.},
	TITLE = {Limit theorems for subcritical branching processes in random
	environment},
	JOURNAL = {Ann. Inst. H. Poincar\'e{} Probab. Statist.},
	FJOURNAL = {Annales de l'Institut Henri Poincar\'e. Probabilit\'es et
	Statistiques},
	VOLUME = {39},
	YEAR = {2003},
	NUMBER = {4},
	PAGES = {593--620},
	ISSN = {0246-0203},
	MRCLASS = {60J80 (60F05 60K37)},
	MRNUMBER = {1983172},
	MRREVIEWER = {Utkir\ Rozikov},
	DOI = {10.1016/S0246-0203(02)00020-1},
	URL = {https://doi.org/10.1016/S0246-0203(02)00020-1},
}

@article {Afanasyev2012,
	AUTHOR = {Afanasyev, V. I. and B\"oinghoff, C. and Kersting, G. and
	Vatutin, V. A.},
	TITLE = {Limit theorems for weakly subcritical branching processes in
	random environment},
	JOURNAL = {J. Theoret. Probab.},
	FJOURNAL = {Journal of Theoretical Probability},
	VOLUME = {25},
	YEAR = {2012},
	NUMBER = {3},
	PAGES = {703--732},
	ISSN = {0894-9840,1572-9230},
	MRCLASS = {60J80 (60F17 60G50 60K37)},
	MRNUMBER = {2956209},
	MRREVIEWER = {Zhun\ Wei\ Lu},
	DOI = {10.1007/s10959-010-0331-6},
	URL = {https://doi.org/10.1007/s10959-010-0331-6},
}

@article {Tanaka1989,
	AUTHOR = {Tanaka, H.},
	TITLE = {Time reversal of random walks in one-dimension},
	JOURNAL = {Tokyo J. Math.},
	FJOURNAL = {Tokyo Journal of Mathematics},
	VOLUME = {12},
	YEAR = {1989},
	NUMBER = {1},
	PAGES = {159--174},
	ISSN = {0387-3870},
	MRCLASS = {60J15},
	MRNUMBER = {1001739},
	MRREVIEWER = {M.\ G.\ Shur},
	DOI = {10.3836/tjm/1270133555},
	URL = {https://doi.org/10.3836/tjm/1270133555},
}

@article {Fu2026,
	AUTHOR = {Fu, W. and Hong, W.},
	TITLE = {On the maximal displacement of critical branching random walk in random environment},
	YEAR = {2026+},
}

@article {Fu2025,
	AUTHOR = {Fu, W. and Hong, W.},
	TITLE = {The exact rate for the tail probability of the maximal
	displacement of subcritical branching random walk},
	JOURNAL = {Electron. J. Probab.},
	FJOURNAL = {Electronic Journal of Probability},
	VOLUME = {30},
	YEAR = {2025},
	PAGES = {--},
	ISSN = {1083-6489},
	MRCLASS = {60J80 (60F17 60G50 60G70)},
	MRNUMBER = {4973692},
	DOI = {10.1214/25-ejp1412},
	URL = {https://doi.org/10.1214/25-ejp1412},
}

@book{B2,
	AUTHOR = {Kersting, G. and  Vatutin, V.},
	TITLE = {Discrete Time Branching Processes in Random Environment},
	YEAR = {2017},
	PUBLISHER = {(J. Wiley Sons, Hoboken, NJ},
}

@article {Addario2009,
	AUTHOR = {Addario-Berry, L. and Reed, B.},
	TITLE = {Minima in branching random walks},
	JOURNAL = {Ann. Probab.},
	FJOURNAL = {The Annals of Probability},
	VOLUME = {37},
	YEAR = {2009},
	NUMBER = {3},
	PAGES = {1044--1079},
	ISSN = {0091-1798,2168-894X},
	MRCLASS = {60J80 (60G50)},
	MRNUMBER = {2537549},
	MRREVIEWER = {Sebastian\ M\"uller},
	DOI = {10.1214/08-AOP428},
	URL = {https://doi.org/10.1214/08-AOP428},
}

@misc{Huang2014,
	title={Branching random walk with a random environment in time}, 
	author={C. Huang and Q. Liu},
	year={2014},
	eprint={1407.7623},
	archivePrefix={arXiv},
	primaryClass={math.PR},
	url={https://arxiv.org/abs/1407.7623}, 
}

@article {Afanasyev2005,
	AUTHOR = {Afanasyev, V. I. and Geiger, J. and Kersting, G. and Vatutin,
	V. A.},
	TITLE = {Criticality for branching processes in random environment},
	JOURNAL = {Ann. Probab.},
	FJOURNAL = {The Annals of Probability},
	VOLUME = {33},
	YEAR = {2005},
	NUMBER = {2},
	PAGES = {645--673},
	ISSN = {0091-1798,2168-894X},
	MRCLASS = {60J80 (60F17 60G50 60K37)},
	MRNUMBER = {2123206},
	MRREVIEWER = {Luis\ G.\ Gorostiza},
	DOI = {10.1214/009117904000000928},
	URL = {https://doi.org/10.1214/009117904000000928},
}

@article {Mallein2019,
	AUTHOR = {Mallein, B. and Mi\l o\'s, P.},
	TITLE = {Maximal displacement of a supercritical branching random walk
	in a time-inhomogeneous random environment},
	JOURNAL = {Stochastic Process. Appl.},
	FJOURNAL = {Stochastic Processes and their Applications},
	VOLUME = {129},
	YEAR = {2019},
	NUMBER = {9},
	PAGES = {3239--3260},
	ISSN = {0304-4149,1879-209X},
	MRCLASS = {60J80 (60K37)},
	MRNUMBER = {3985561},
	MRREVIEWER = {David\ A.\ Croydon},
	DOI = {10.1016/j.spa.2018.09.008},
	URL = {https://doi.org/10.1016/j.spa.2018.09.008},
}

@book {LDP1998,
	AUTHOR = {Dembo, A. and Zeitouni, O.},
	TITLE = {Large deviations techniques and applications},
	SERIES = {Applications of Mathematics (New York)},
	VOLUME = {38},
	EDITION = {Second},
	PUBLISHER = {Springer-Verlag, New York},
	YEAR = {1998},
	PAGES = {xvi+396},
	ISBN = {0-387-98406-2},
	MRCLASS = {60F10},
	MRNUMBER = {1619036},
	DOI = {10.1007/978-1-4612-5320-4},
	URL = {https://doi.org/10.1007/978-1-4612-5320-4},
}

@article {Kesten1995,
	AUTHOR = {Kesten, H.},
	TITLE = {Branching random walk with a critical branching part},
	JOURNAL = {J. Theoret. Probab.},
	FJOURNAL = {Journal of Theoretical Probability},
	VOLUME = {8},
	YEAR = {1995},
	NUMBER = {4},
	PAGES = {921--962},
	ISSN = {0894-9840,1572-9230},
	MRCLASS = {60J15 (60J80)},
	MRNUMBER = {1353560},
	MRREVIEWER = {David\ J.\ Aldous},
	DOI = {10.1007/BF02410118},
	URL = {https://doi.org/10.1007/BF02410118},
}

@article {Biggins2004,
	AUTHOR = {Biggins, J. D. and Kyprianou, A. E.},
	TITLE = {Measure change in multitype branching},
	JOURNAL = {Adv. in Appl. Probab.},
	FJOURNAL = {Advances in Applied Probability},
	VOLUME = {36},
	YEAR = {2004},
	NUMBER = {2},
	PAGES = {544--581},
	ISSN = {0001-8678,1475-6064},
	MRCLASS = {60J80 (60G42)},
	MRNUMBER = {2058149},
	MRREVIEWER = {Owen\ Jones},
	DOI = {10.1239/aap/1086957585},
	URL = {https://doi.org/10.1239/aap/1086957585},
}

@article {Fleischman1979Maximun,
	AUTHOR = { Sawyer, S. and  Fleischman, J.},
	TITLE = {Maximum geographic range of a mutant allele considered as a subtype of a Brownian branching random field},
	JOURNAL = {PNAS},
	FJOURNAL = {Proceedings of the National Academy of Sciences of the United States of America},
	VOLUME = {76},
	YEAR = {1979},
	PAGEs = {872-875},
	NUMBER = {2},
}

@article {Neuman2017,
	AUTHOR = {Neuman, E. and Zheng, X.},
	TITLE = {On the maximal displacement of subcritical branching random
	walks},
	JOURNAL = {Probab. Theory Related Fields},
	FJOURNAL = {Probability Theory and Related Fields},
	VOLUME = {167},
	YEAR = {2017},
	NUMBER = {3-4},
	PAGES = {1137--1164},
	ISSN = {0178-8051,1432-2064},
	MRCLASS = {60J80 (60G70)},
	MRNUMBER = {3627435},
	MRREVIEWER = {Edward\ C.\ Waymire},
	DOI = {10.1007/s00440-016-0702-8},
	URL = {https://doi.org/10.1007/s00440-016-0702-8},
}

@article {Lalley2015,
	AUTHOR = {Lalley, S. P. and Shao Y.},
	TITLE = {On the maximal displacement of critical branching random walk},
	JOURNAL = {Probab. Theory Related Fields},
	FJOURNAL = {Probability Theory and Related Fields},
	VOLUME = {162},
	YEAR = {2015},
	NUMBER = {1-2},
	PAGES = {71--96},
	ISSN = {0178-8051,1432-2064},
	MRCLASS = {60J80},
	MRNUMBER = {3350041},
	MRREVIEWER = {Manuel\ Molina},
	DOI = {10.1007/s00440-014-0566-8},
	URL = {https://doi.org/10.1007/s00440-014-0566-8},
}

@article {Hu2009,
	AUTHOR = {Hu, Y. and Shi, Z.},
	TITLE = {Minimal position and critical martingale convergence in
	branching random walks, and directed polymers on disordered
	trees},
	JOURNAL = {Ann. Probab.},
	FJOURNAL = {The Annals of Probability},
	VOLUME = {37},
	YEAR = {2009},
	NUMBER = {2},
	PAGES = {742--789},
	ISSN = {0091-1798,2168-894X},
	MRCLASS = {60J80 (60G42)},
	MRNUMBER = {2510023},
	MRREVIEWER = {David\ A.\ Croydon},
	DOI = {10.1214/08-AOP419},
	URL = {https://doi.org/10.1214/08-AOP419},
}

@article {Bramson2009,
	AUTHOR = {Bramson, M. and Zeitouni, O.},
	TITLE = {Tightness for a family of recursion equations},
	JOURNAL = {Ann. Probab.},
	FJOURNAL = {The Annals of Probability},
	VOLUME = {37},
	YEAR = {2009},
	NUMBER = {2},
	PAGES = {615--653},
	ISSN = {0091-1798,2168-894X},
	MRCLASS = {60J80 (60G50)},
	MRNUMBER = {2510018},
	MRREVIEWER = {Quansheng\ Liu},
	DOI = {10.1214/08-AOP414},
	URL = {https://doi.org/10.1214/08-AOP414},
}

@article {Bramson2016,
	AUTHOR = {Bramson, M. and Ding, J. and Zeitouni, O.},
	TITLE = {Convergence in law of the maximum of nonlattice branching
	random walk},
	JOURNAL = {Ann. Inst. Henri Poincar\'{e} Probab. Stat.},
	FJOURNAL = {Annales de l'Institut Henri Poincar\'{e} Probabilit\'{e}s et
	Statistiques},
	VOLUME = {52},
	YEAR = {2016},
	NUMBER = {4},
	PAGES = {1897--1924},
	ISSN = {0246-0203,1778-7017},
	MRCLASS = {60G70 (60J80)},
	MRNUMBER = {3573300},
	MRREVIEWER = {Anja\ K.\ Sturm},
	DOI = {10.1214/15-AIHP703},
	URL = {https://doi.org/10.1214/15-AIHP703},
}

@article {Bachmann,
	AUTHOR = {Bachmann, M.},
	TITLE = {Limit theorems for the minimal position in a branching random
	walk with independent logconcave displacements},
	JOURNAL = {Adv. in Appl. Probab.},
	FJOURNAL = {Advances in Applied Probability},
	VOLUME = {32},
	YEAR = {2000},
	NUMBER = {1},
	PAGES = {159--176},
	ISSN = {0001-8678,1475-6064},
	MRCLASS = {60J80 (60F05 60G70)},
	MRNUMBER = {1765165},
	MRREVIEWER = {Karl\ Grill},
	DOI = {10.1239/aap/1013540028},
	URL = {https://doi.org/10.1239/aap/1013540028},
}

@article {Aidekon2013,
	AUTHOR = {A\"{\i}d\'{e}kon, E.},
	TITLE = {Convergence in law of the minimum of a branching random walk},
	JOURNAL = {Ann. Probab.},
	FJOURNAL = {The Annals of Probability},
	VOLUME = {41},
	YEAR = {2013},
	NUMBER = {3A},
	PAGES = {1362--1426},
	ISSN = {0091-1798,2168-894X},
	MRCLASS = {60J80 (60F05)},
	MRNUMBER = {3098680},
	MRREVIEWER = {P\'{e}ter\ Kevei},
	DOI = {10.1214/12-AOP750},
	URL = {https://doi.org/10.1214/12-AOP750},
}

@article {Bramson,
	AUTHOR = {Bramson, M. D.},
	TITLE = {Minimal displacement of branching random walk},
	JOURNAL = {Z. Wahrsch. Verw. Gebiete},
	FJOURNAL = {Zeitschrift f\"{u}r Wahrscheinlichkeitstheorie und Verwandte
	Gebiete},
	VOLUME = {45},
	YEAR = {1978},
	NUMBER = {2},
	PAGES = {89--108},
	ISSN = {0044-3719},
	MRCLASS = {60J80},
	MRNUMBER = {510529},
	MRREVIEWER = {J.\ D.\ Biggins},
	DOI = {10.1007/BF00715186},
	URL = {https://doi.org/10.1007/BF00715186},
}

@article {Biggins1976,
	AUTHOR = {Biggins, J. D.},
	TITLE = {The first- and last-birth problems for a multitype
	age-dependent branching process},
	JOURNAL = {Advances in Appl. Probability},
	FJOURNAL = {Advances in Applied Probability},
	VOLUME = {8},
	YEAR = {1976},
	NUMBER = {3},
	PAGES = {446--459},
	ISSN = {0001-8678,1475-6064},
	MRCLASS = {60J80},
	MRNUMBER = {420890},
	MRREVIEWER = {Norman\ Kaplan},
	DOI = {10.2307/1426138},
	URL = {https://doi.org/10.2307/1426138},
}

@article {Kingman,
	AUTHOR = {Kingman, J. F. C.},
	TITLE = {The first birth problem for an age-dependent branching
	process},
	JOURNAL = {Ann. Probability},
	FJOURNAL = {The Annals of Probability},
	VOLUME = {3},
	YEAR = {1975},
	NUMBER = {5},
	PAGES = {790--801},
	ISSN = {0091-1798},
	MRCLASS = {60J80},
	MRNUMBER = {400438},
	MRREVIEWER = {H.\ Kesten},
	DOI = {10.1214/aop/1176996266},
	URL = {https://doi.org/10.1214/aop/1176996266},
}

@article {Hammersley,
	AUTHOR = {Hammersley, J. M.},
	TITLE = {Postulates for subadditive processes},
	JOURNAL = {Ann. Probability},
	FJOURNAL = {The Annals of Probability},
	VOLUME = {2},
	YEAR = {1974},
	PAGES = {652--680},
	ISSN = {0091-1798},
	MRCLASS = {60G10 (60F05 60J80)},
	MRNUMBER = {370721},
	MRREVIEWER = {H.\ Kesten},
	DOI = {10.1214/aop/1176996611},
	URL = {https://doi.org/10.1214/aop/1176996611},
}

\end{document}